\documentclass[11pt, letterpaper]{amsart}
\usepackage[plainpages=false,pdfpagelabels,
colorlinks=true,linkcolor=blue,
citecolor=blue,filecolor=blue,      
urlcolor=blue,hypertexnames=false]{hyperref}
\usepackage{amssymb}
\usepackage{graphics}
\usepackage{graphicx}               
\usepackage{latexsym}
\usepackage{amsmath}
\usepackage{amssymb,amsthm,amsfonts}
\usepackage{amscd}
\usepackage[arrow, matrix, curve]{xy}
\usepackage{syntonly}
\usepackage{tikz}
\usepackage[small,nohug,heads=vee]{diagrams}
\diagramstyle[labelstyle=\scriptstyle]
\usepackage{cancel}
\makeatletter
\@namedef{subjclassname@2020}{\textup{2020} Mathematics Subject Classification}
\makeatother

\theoremstyle{plain}
\newtheorem{thm}{Theorem}[section]
\newtheorem{theorem}[thm]{Theorem}
\newtheorem*{theorem*}{Theorem}
\newtheorem{lemma}[thm]{Lemma}
\newtheorem{corollary}[thm]{Corollary}
\newtheorem{proposition}[thm]{Proposition}

\theoremstyle{definition}

\newtheorem{example}[thm]{Example}
\newtheorem{conjecture}[thm]{Conjecture}
\numberwithin{equation}{thm}

\title[Chevalley restriction theorem for orthogonal groups ]{A higher-dimensional Chevalley restriction theorem for orthogonal groups}

\author[]{Lei Song}
\address{School of Mathematics, Sun Yat-sen University\\No. 135 Xingang Xi Road, Guangzhou, Guangdong 510275, P. R. China}
\email{songlei3@mail.sysu.edu.cn}
\author[]{Xiaopeng Xia}
\address{School of Mathematical Sciences, University of Science and Technology of China\\ No. 96 Jinzhai Road, Hefei, Anhui 230026, P. R. China}
\email{xpxia@mail.ustc.edu.cn}
\author[]{Jinxing Xu}%
\address{School of Mathematical Sciences, University of Science and Technology of China\\ No. 96 Jinzhai Road, Hefei, Anhui 230026, P. R. China}
\email{xujx02@ustc.edu.cn}%

\date{}

\begin{document}
\subjclass[2020]{14L30 (primary), 20G05 (secondary).}
  \keywords{Chevalley Restriction Theorem, Invariant Theory, Commuting Scheme, Orthogonal Groups, Pfaffians}
 \begin{abstract}
 We prove a higher-dimensional Chevalley restriction theorem for orthogonal groups, which was conjectured by Chen and Ng\^{o} for reductive groups.  In characteristic $p>2$, we also prove a weaker statement. In characteristic $0$, the theorem implies that the categorical quotient of a commuting scheme by the diagonal adjoint action of the group is integral and normal. As applications, we deduce some trace identities and a certain multiplicative property of the Pfaffian over an arbitrary commutative algebra.
 \end{abstract}
  \maketitle

\section{Introduction}

Let $G$ be a reductive group over an algebraically closed field $\mathbb{K}$ with Lie algebra $\mathfrak{g}$. For an integer $d\geq 2$,  let  $\mathfrak{C}^d_{\mathfrak{g}}\subset \mathfrak{g}^d$ be the  commuting scheme, which is defined as the scheme-theoretic fiber of the commutator map over the zero
\[\mathfrak{g}^d\rightarrow \prod\limits_{i<j} \mathfrak{g}, \ \ (x_1,\cdots, x_d)\mapsto \prod_{i<j}[x_i, x_j].\] 
Its underlying variety (the reduced induced closed subscheme) $\mathfrak{C}^d_{\mathfrak{g},red}$ is called the commuting variety. As a set, $\mathfrak{C}^d_{\mathfrak{g},red}$ consists of $d$-tuples $(x_1,\cdots, x_d)\in \mathfrak{g}^d$ such that $[x_i,x_j]=0$, for all  $1\leq i,j\leq d$. It is a long-standing open question whether or not $\mathfrak{C}^d_{\mathfrak{g}}$ is reduced, that is $\mathfrak{C}^d_{\mathfrak{g}}=\mathfrak{C}^d_{\mathfrak{g},red}$.  When the characteristic of $\mathbb{K}$ ($\rm{char} \ \mathbb{K}$ for short) is zero, Charbonnel [\ref{Charbonnel}] recently claims a proof for $\mathfrak{C}^2_{\mathfrak{g}}=\mathfrak{C}^2_{\mathfrak{g},red}$. Although there is no adequate evidence to expect that $\mathfrak{C}^d_{\mathfrak{g}}$ is reduced for general $d$, one can study the categorical quotient $\mathfrak{C}^d_{\mathfrak{g}}/\!\!/G$ and ask the same question. Here $G$ acts on $\mathfrak{g}^d$ by the diagonal adjoint action, and the action leaves $\mathfrak{C}^d_{\mathfrak{g}}$ stable. 

Let $T$ be a maximal torus of $G$ and $\mathfrak{t}$ be the Lie algebra of $T$. Then the Weyl group $W:=N_G(T)/T$ acts on $\mathfrak{t}^d$ diagonally. The embedding $\mathfrak{t}^d\hookrightarrow \mathfrak{g}^d$ factors through $\mathfrak{C}^d_{\mathfrak{g}}$ and induces the natural morphism 
\[\Phi: \mathfrak{t}^d/\!\!/W \rightarrow \mathfrak{C}^d_{\mathfrak{g}}/\!\!/ G.\]
In studying the Hitchin morphism from the moduli stack of principal $G$-Higgs bundles on a proper smooth variety $X$ of dimension $d\ge 2$, Chen and Ng\^{o} \cite{Chen-Ngo} are led to

\begin{conjecture}[Chen-Ng\^{o}]\label{conj}
The morphism $\Phi: \mathfrak{t}^d/\!\!/W\rightarrow \mathfrak{C}^d_{\mathfrak{g}}/\!\!/G$
is an isomorphism. 
\end{conjecture}
 When $d=1$ and $\rm{char} \ \mathbb{K} =0$, Conjecture \ref{conj} is simply the classical Chevalley restriction theorem. Since in the context of Higgs bundles, $d$ is the dimension of the underlying variety $X$, we view the conjecture as a higher-dimensional analogue of Chevalley restriction theorem. Note when $d=2$ and $\rm{char} \ \mathbb{K} =0$, this conjecture is a special  case (degree zero part) of a more general conjecture proposed by Berest et al. [\ref{Berest et al.}].


If $\rm{char} \ \mathbb{K} =0$, Conjecture \ref{conj} is  known  to hold for  $G=GL_n(\mathbb{K})$ (Vaccarino [\ref{Vaccarino}], Domokos [\ref{Domokos}], and later Chen-Ng\^{o} [\ref{Chen-Ngo}] independently; see also Gan-Ginzburg [\ref{Gan-Ginzburg}] for case $d=2$) and for $G=Sp_{n}(\mathbb{K})$ (Chen-Ng\^{o} [\ref{Chen-Ngo-Symplectic}]). 
A weaker version $\mathfrak{t}^d/\!\!/W\xrightarrow{\sim} \mathfrak{C}^d_{\mathfrak{g},red}/\!\!/G$ is proved by Hunziker [\ref{Hunziker}] if $G$ is of type $A,B,C,D$ or $G_2$.

If $\rm{char}\ \mathbb{K} >0$, Conjecture \ref{conj} is largely open. However, Vaccarino [\ref{Vaccarino}] proved the weaker version $\mathfrak{t}^d/\!\!/W\xrightarrow{\sim} \mathfrak{C}^d_{\mathfrak{g},red}/\!\!/G$ for $G=GL_n(\mathbb{K})$.   

The main purpose of the article is to prove Conjecture \ref{conj} for orthogonal groups in case $\rm{char} \ \mathbb{K} =0$ and to prove a weaker version of this conjecture in case $\rm{char} \ \mathbb{K}>2$. To be more  precise, our main result is the following (see Theorems \ref{main thm: char.  0}, \ref{thm:char.  p case}, \ref{thm:main thm in the n even G=SO_n case} and \ref{thm:n even G=SO_n, char.  p case}):
\begin{theorem}\label{main thm in introduction}
Suppose $n\geq 2$, $d\geq 1$, and  $G$ is an orthogonal group $O_n(\mathbb{K})$ or a special orthogonal group $SO_n(\mathbb{K})$. 
\begin{enumerate}
\item If $\rm{char} \ \mathbb{K} =0$, then $\Phi: \mathfrak{t}^d/\!\!/W\rightarrow \mathfrak{C}^d_{\mathfrak{g}}/\!\!/G$ is an isomorphism.
\item If $\rm{char} \ \mathbb{K} >2$, then $\Phi: \mathfrak{t}^d/\!\!/W\rightarrow \mathfrak{C}^d_{\mathfrak{g},red}/\!\!/G$ is an isomorphism. 
\end{enumerate} 
\end{theorem}
Our proof can also treat in a uniform way the case $G=Sp_n(\mathbb{K})$ ($n$ even), which is due to Chen-Ng\^{o} [\ref{Chen-Ngo-Symplectic}] if $\rm{char} \ \mathbb{K} =0$.  

Since  $W$ is finite and $\mathfrak{t}^d$ is an affine space,  the theorem implies
\begin{corollary}
If $\rm{char} \ \mathbb{K} =0$ and $G=O_n(\mathbb{K})$ or $SO_n(\mathbb{K})$,  the quotient $\mathfrak{C}^d_{\mathfrak{g}}/\!\!/G$ is integral (i.e. reduced and irreducible) and  normal. $\hfill$ $\qed$
\end{corollary}

We will divide the proof into two parts, according to the type of the Weyl group $W$. In the first part, $W$ is of type $B$, and this includes cases $G=O_n(\mathbb{K})$, $SO_n(\mathbb{K})$ ($n$ odd), and $Sp_n(\mathbb{K})$ ($n$ even). Based on the results in the first part (Theorems \ref{main thm: char.  0} and \ref{thm:char.  p case}), we prove Theorem \ref{thm:main thm in the n even G=SO_n case} and Theorem \ref{thm:n even G=SO_n, char.  p case} in the second one, which corresponds to the $W$ of type $D$ ($G=SO_n(\mathbb{K})$, $n$ even).

Next let us explain the strategy of  our proof of Theorem \ref{main thm: char.  0}.  Let $\mathbb{K}[\mathfrak{C}^d_{\mathfrak{g}}]$ (resp. $\mathbb{K}[\mathfrak{t}^d]$) be the coordinate ring of $\mathfrak{C}^d_{\mathfrak{g}}$  (resp. $\mathfrak{t}^d$).
In order to show the restriction homomorphism $\Phi: \mathbb{K}[\mathfrak{C}_{\mathfrak{g}}^d]^G\rightarrow \mathbb{K}[\mathfrak{t}^d]^W$ is an isomorphism, it suffices to construct a $\mathbb{K}$-linearly spanning set of $\mathbb{K}[\mathfrak{C}_{\mathfrak{g}}^d]^G$, which are mapped bijectively to a $\mathbb{K}$-linear basis of $\mathbb{K}[\mathfrak{t}^d]^W$. 

The results of Procesi [\ref{procesi:matrix invariants}] give a set of generators of $\mathbb{K}[M_n(\mathbb{K})^d]^G $, which induces a set of generators of $\mathbb{K}[\mathfrak{C}^d_{\mathfrak{g}}]^G$ under the surjective homomorphism $\mathbb{K}[M_n(\mathbb{K})^d]^G \rightarrow\mathbb{K}[\mathfrak{C}^d_{\mathfrak{g}}]^G$. We  shall encode these generators in the determinant of a universal matrix.  In \eqref{equ:defn of F_g and   F_t}, we define the formal power series
\begin{displaymath}
F_{\mathfrak{g}}=\det (I_n+\sum\limits_{(i_1,\cdots, i_d)\in S}X(1)^{i_1}X(2)^{i_2}\cdots X(d)^{i_d}T_{i_1\cdots i_d}),
\end{displaymath}
with $F_{\mathfrak{g}}\in \mathbb{K}[\mathfrak{C}^d_{\mathfrak{g}}]^G[[T_{i_1\cdots i_d}|(i_1,\cdots, i_d)\in S]]$, where the index set $S$ is
\[\{(i_1, \cdots, i_d)\:|\: i_1, \cdots, i_d \in \mathbb{Z}_{\geq 0}, (i_1, \cdots, i_d)\neq (0,\cdots, 0), i_1+\cdots + i_d \textmd{ is even}\}, \] 
and $X(1),\cdots, X(d)$ are ``generic" commuting skew symmetric $n\times n$ matrices over $\mathbb{K}[\mathfrak{C}^d_{\mathfrak{g}}]$.
Based on Procesi's results, we can show that a $\mathbb{K}$-linearly spanning set of $\mathbb{K}[\mathfrak{C}^d_{\mathfrak{g}}]^G$ is given by the coefficients of $F_{\mathfrak{g}}$, and in turn  by the coefficients of $\sqrt{F_{\mathfrak{g}}}$, the unique square root of $F_{\mathfrak{g}}$ with constant coefficient $1$. 

Under the restriction homomorphism $\Phi: \mathbb{K}[\mathfrak{C}_{\mathfrak{g}}^d]^G\rightarrow \mathbb{K}[\mathfrak{t}^d]^W$, we deduce that
\[\Phi(\sqrt{F_{\mathfrak{g}}})=\sqrt{F_{\mathfrak{t}}}.\]
See \eqref{equ:defn of F_g and   F_t}, \eqref{equ:expansion F_t, F_g} for the definition of $F_{\mathfrak{t}}$ and $\sqrt{F_{\mathfrak{t}}}$.
 Now explicit computations show $\deg \sqrt{F_{\mathfrak{t}}}\leq \lfloor\frac{n}{2}\rfloor$, that is,  $\sqrt{F_{\mathfrak{t}}}$ can be written as  
 \[\sqrt{F_{\mathfrak{t}}}=1+ \sum\limits_{w\in \mathcal{M}(\lfloor\frac{n}{2}\rfloor)} c_w w,\] 
where $c_w\in \mathbb{K}[\mathfrak{t}^d]^W$ and $\mathcal{M}(\lfloor\frac{n}{2}\rfloor)$ is the set of non-empty monomials in  variables $\{T_{i_1 \cdots i_d}| (i_1, \cdots, i_d)\in S\}$ whose degree is less than or equal to $\lfloor\frac{n}{2}\rfloor$. We can also show that the coefficients of  $\sqrt{F_{\mathfrak{t}}}$ form a $\mathbb{K}$-linear basis of $\mathbb{K}[\mathfrak{t}^d]^W$.  So to finish, all we need to show is $\deg \sqrt{F_{\mathfrak{g}}}\leq \lfloor\frac{n}{2}\rfloor$.  This turns out to be the technical heart of this paper.     

To prove $\deg \sqrt{F_{\mathfrak{g}}}\leq \lfloor\frac{n}{2}\rfloor$, we explicitly construct a power series  $N$ with constant coefficient $1$ such that $N^2=F_{\mathfrak{g}}$. By the uniqueness of square roots with constant coefficient $1$, we see $N=\sqrt{F_{\mathfrak{g}}}$.  So the degree bound follows from the explicit construction of $N$.  

If $G=Sp_n(\mathbb{K})$, our construction of $N$ is essentially the same as that presented in Chen-Ng\^{o} [\ref{Chen-Ngo-Symplectic}]. However, if $G=O_n(\mathbb{K})$, the construction of $N$ is much more involved and poses significant challenges. We provide the full details in Section \ref{subsec:construction of N}.

\subsection*{Acknowledgments}
L. S. thanks Jianxun Hu for the encouragement. J. X. would like to thank Xiaowen Hu and Mao Sheng for helpful discussions related to this work. The authors would like to thank Hao Sun for the helpful comments which improve the manuscript. We are very grateful to the anonymous referee for a careful reading  and for providing valuable references. During the preparation of the article, 
L. S. was partially supported by the Guangdong Basic and Applied Basic Research Foundation 2020A1515010876, X. X. was partially supported by the Innovation Program for Quantum Science and Technology (2021ZD0302902), and J. X. was partially supported by the National Key R and D Program of China (2020YFA0713100),   CAS Project for Young Scientists in Basic Research (YSBR-032) and NSFC (12271495).

\section{Notations and preliminaries}\label{sec:some lemmas}

In this section we fix some notations and record some useful lemmas that will be used frequently in the subsequent sections.  

Throughout the paper,   $\mathbb{K}$ is an algebraically closed field of characteristic not equal to $2$. All rings are commutative, unless otherwise specified.

Given a ring $R$, we denote by $M_n(R)$ the set of $n\times n$ matrices over  $R$, and for a matrix $M$, we denote by $M^t$ its transpose, and by $M(i, j)$ the $(i, j)$-entry of $M$. We denote the coordinate ring of an affine $\mathbb{K}$-scheme $X$ by $\mathbb{K}[X]$. In particular, 
if $V$ is a $\mathbb{K}$-linear space, $\mathbb{K}[V]$ means the $\mathbb{K}$-algebra of polynomial functions on $V$. If a group $G$ acts $\mathbb{K}$-linearly on $V$, then $G$ acts naturally  on $\mathbb{K}[V]$ by $g\cdot f \ (v) =f(g^{-1}\cdot v)$, for $g\in G, \ f\in \mathbb{K}[V], \ v\in V$.  The $\mathbb{K}$-algebra  of $G$-invariant polynomials on $V$ is denoted by  $\mathbb{K}[V]^G$.  For $n\in \mathbb{Z}$, we denote by $\lfloor n \rfloor$ the maximal integer less than or equal to $n$.

Here are some lemmas on square roots of formal power series. 

\begin{lemma}\label{lemma:f^2 determines f, first form}
 Suppose $R=\oplus_{i=0}^{\infty} R_i$ is a  graded $\mathbb{K}$-algebra. Let $f_1, f_2\in R$ and let $c_j\in R_0$ be the degree zero part of $f_j$, $j=1,2$. If $f_1^2=f_2^2$, $c_1=c_2$, and if $c_1$ is not a zero-divisor in $R$, then $f_1=f_2$.
\end{lemma}
\begin{proof}
Put $f_1=c_1+g_1$ and  $f_2=c_2+g_2$. Since $f_1^2=f_2^2$ and $c_1=c_2$, we see 
  $$
(c_1+c_2+g_1+g_2) (g_1-g_2)=0.
  $$
Since $c_1+c_2=2c_1$ is not a zero-divisor in $R$, it is easy to see $c_1+c_2+g_1+g_2$ is not a zero-divisor in $R$. From this, we deduce $g_1=g_2$ and hence $f_1=f_2$. 
\end{proof}

\begin{lemma}\label{lemma:f^2 determines f}
Let  $R$ be a  $\mathbb{K}$-algebra, and  $R[[t_1, t_2,\cdots]]$ be the formal power series ring  with countable  variables.
\begin{enumerate}
\item Suppose $f_1, f_2\in R[[t_1, t_2,\cdots]]$, and let $c_j=f_j(0,0,\cdots )\in R$ be the constant coefficient  of $f_j$. If $f_1^2=f_2^2$, $c_1=c_2$, and if $c_1$ is not a zero-divisor in $R$, then $f_1=f_2$.
\item Suppose $g\in R[[t_1, t_2,\cdots]]$ has constant coefficient  $1$, then there exists a unique $f\in R[[t_1, t_2,\cdots]]$ whose constant coefficient is $1$ and which satisfies $f^2=g$. Moreover, the $\mathbb{K}$-subalgebra of $R$ generated by the coefficients of $g$ coincides with  that generated by the coefficients of $f$. 
\end{enumerate}

\end{lemma}
\begin{proof}
\begin{enumerate}
\item The proof is the same as that for the previous lemma.
\item The uniqueness of $f$ follows from $(1)$. In order to show its existence, let $\mathcal{T}$ be the set of all nonempty monomials in the variables $\{t_1, t_2,\cdots\}$. For $w_1,\ w_2\in \mathcal{T}$, define $w_1\leq w_2$ if there exists $w\in \{1\}\cup \mathcal{T}$ such that $w_2=w_1 w$. This gives a  partial order on $\mathcal{T}$.  Write $g=1+\sum_{w\in \mathcal{T}} a_w w$, with $a_w\in R$. For each $w\in \mathcal{T}$, we will define  by induction $b_w\in R$, so that $f:=1+\sum_{w\in \mathcal{T}}b_w w$ satisfies $f^2=g$. 

Suppose $w\in \mathcal{T}$ and suppose for any $v\in \mathcal{T}$ with $v<w$, the element $b_v\in R$ has been defined. Let $\mathcal{T}_w:=\{(v_1, v_2)\in \mathcal{T}^2| v_1<w, \ v_2<w, v_1 v_2=w\}$.   Then define 
\[b_w:=\frac{1}{2}(a_w-\sum_{(v_1,v_2)\in \mathcal{T}_w}b_{v_1}b_{v_2}).\]
By direct computations, we can see that  $f:=1+\sum_{w\in \mathcal{T}}b_w w$ satisfies $f^2=g$. Moreover, it follows from the explicit expressions that the $\mathbb{K}$-subalgebras $\mathbb{K}[a_w| w\in \mathcal{T}]$ and $\mathbb{K}[b_w| w\in \mathcal{T}]$ of $R$ coincide. 
\end{enumerate}
 \end{proof}

\begin{lemma}\label{lemma:f^2=y^2n g implies f=y^n h}
Let $R=\oplus_{i=0}^{\infty} R_i$ be a  graded $\mathbb{K}$-algebra, with $R_0=\mathbb{K}$.  Suppose $f, g\in R[[y,z_0,z_1,\cdots]]$ satisfy $f^2=y^{2n}g$  and  $g=z_0^{2m}+g_1+g_2+\cdots$,  for $ g_i\in R_i[[y,z_0,z_1,\cdots]]$, with  $n, m \geq 0$, then there exists $\tilde{f}\in R[[y,z_0,z_1,\cdots]]$ satisfying $f=y^n \tilde{f}$. Moreover, if $f^2=y^{2n}$, then $f=\pm y^n$.
\end{lemma}
\begin{proof}
Let $f=f_0+f_1+\cdots $,  with $f_i\in R_i[[y,z_0,z_1,\cdots]]$. By $f^2=y^{2n} g$  we get $f_0^2=y^{2n}z_0^{2m}$ in $R_0[[y,z_0,z_1,\cdots]]=\mathbb{K}[[y,z_0,z_1,\cdots]]$. This implies  $f_0=\pm y^nz_0^m$.   Assume without loss of generality that $f_0=y^nz_0^m$.  Now suppose $i\geq 1$ and suppose for each $1\leq j<i$, there exists $\tilde{f}_j\in R_j[[y,z_0,z_1,\cdots]]$ such that $f_j=y^n \tilde{f}_j$.  By comparing the $R_i[[y,z_0,z_1,\cdots]]$ part of $f^2$ and $y^{2n}g$, we get $\sum_{j=0}^i f_j f_{i-j}=y^{2n} g_i$. Since $\sum_{j=0}^i f_j f_{i-j}=2f_0f_i+\sum_{j=1}^{i-1}f_jf_{i-j}=2y^nz_0^m f_i+y^{2n}\sum_{j=1}^{i-1}\tilde{f}_j\tilde{f}_{i-j}$, we obtain $2y^nz_0^m f_i=y^{2n} (g_i-\sum_{j=1}^{i-1}\tilde{f}_j\tilde{f}_{i-j})$. As neither $z_0$ nor $y$ is a zero-divisor in $R[[y,z_0,z_1,\cdots]]$, this implies $2z_0^m f_i=y^{n} (g_i-\sum_{j=1}^{i-1}\tilde{f}_j\tilde{f}_{i-j})$, and then $f_i=y^n\tilde{f}_i$ for some $\tilde{f}_i\in R_i[[y,z_0,z_1,\cdots]]$. By induction on $i$ we see that for each $i\geq 1$, 
there exists $\tilde{f}_i\in R_i[[y,z_0,z_1,\cdots]]$ satisfying $f_i=y^n\tilde{f}_i$. Since $f\in R[[y,z_0,z_1,\cdots]]$, the formal sum $\tilde{f}:=z_0^m+\sum_{i=1}^{\infty}\tilde{f}_i$ can be viewed as an element of $R[[y,z_0,z_1,\cdots]]$ and it satisfies $f=y^n\tilde{f}$.

If $f^2=y^{2n}$, then $\tilde{f}^2=1$, and this implies $\tilde{f}\in R_0=\mathbb{K}$. So $\tilde{f}=\pm1$ and $f=\pm y^n$. 
\end{proof}

\section{Main theorems: type $B$ case}\label{sec:main theorems}
Suppose $n\geq 2, \ d\geq 1$ are   positive integers. Let 
\begin{displaymath}
O_n(\mathbb{K}):=\{A\in M_n(\mathbb{K})\:|\: A A^t=I_n\}
\end{displaymath}
be the orthogonal group, and if $n$ is even, let
\begin{displaymath}
Sp_n(\mathbb{K}):=\{A\in M_n(\mathbb{K})\:|\: A^t J A =I_n\}
\end{displaymath}
be the symplectic group, where $J=\begin{pmatrix}
0& I_{\frac{n}{2}}\\ 
-I_{\frac{n}{2}}&0
\end{pmatrix}$. 
 
 Throughout this section, $G$ is  one of the following groups:
\begin{displaymath}
G=
\begin{cases}
O_n(\mathbb{K}), \\ 
Sp_n(\mathbb{K}), & n \textmd{   even},\\ 
SO_n(\mathbb{K}), & n \textmd{   odd}.
\end{cases}
\end{displaymath}

Let $T$ be a maximal torus of $G$ and $W=N_G(T)/T$ be the Weyl group.  If $n\geq 4$ is even, the Weyl groups of $O_n(\mathbb{K})$ and $Sp_n(\mathbb{K})$ are isomorphic, and both are of type $B$.  As usual,  $\mathfrak{g}$ denotes the Lie algebra of $G$. Explicitly, 
\begin{displaymath}
\mathfrak{g}=\begin{cases}
\mathfrak{so}_n, \ \textmd{ if } G \ \textmd{ is } O_n(\mathbb{K}) \textmd{ or } SO_n(\mathbb{K}); \\ 
\mathfrak{sp}_n, \ \textmd{ if } G \ \textmd{ is } Sp_n(\mathbb{K}),
\end{cases}
\end{displaymath}
where we fix the realizations of the Lie algebras as matrices:
\begin{eqnarray*}
\mathfrak{so}_n&= & \{A\in M_n(\mathbb{K}\:)|\: A+A^t=0\}, \\
\mathfrak{sp}_n &= &\{A\in M_n(\mathbb{K})\:|\: A^t J+ J A=0\}.
\end{eqnarray*}

The Lie algebra $\mathfrak{t}$ of $T$ is  a Cartan subalgebra of $\mathfrak{g}$. Via the diagonal adjoint  representation, the group $G$ acts on $\mathfrak{g}^d$ and this induces  an action of $W$ on $\mathfrak{t}^d$. For $1\leq k\leq d$, \  $1\leq i,j\leq n$, let $x(k)_{ij}$ be the polynomial function of $\mathfrak{g}^d$ whose value at a point $(A_1,\cdots, A_d)\in \mathfrak{g}^d$ is the $(i,j)$-entry of the matrix $A_k\in M_n(\mathbb{K})$. 

 Over the  ring $\mathbb{K}[\mathfrak{g}^d]$, consider the ``generic" $n\times n$ matrices $X(1)$, $X(2)$, $\cdots$, $X(d)$, such that  the $(i,j)$-entry of $X(k)$ is $x(k)_{ij}$. 
 Let $I$ be the ideal of $\mathbb{K}[\mathfrak{g}^d]$ generated by all of the entries of the matrices $[X(k), X(l)]:=X(k)X(l)-X(l)X(k), \ 1\leq k, l\leq d$, and define the quotient ring 
 \[\mathbb{K}[\mathfrak{C}^d_{\mathfrak{g}}]:=\mathbb{K}[\mathfrak{g}^d]/I. \]
 This ring can be viewed as the coordinate ring of the commuting scheme $\mathfrak{C}^d_{\mathfrak{g}}$. Since $I$ is obviously a homogeneous ideal of the polynomial ring $\mathbb{K}[\mathfrak{g}^d]$, the quotient ring  $\mathbb{K}[\mathfrak{C}^d_{\mathfrak{g}}]=\oplus_{i=0}^{\infty}\mathbb{K}[\mathfrak{C}^d_{\mathfrak{g}}]_i$ is a graded $\mathbb{K}$-algebra. Moreover, the  degree zero part $\mathbb{K}[\mathfrak{C}^d_{\mathfrak{g}}]_0=\mathbb{K}$, and the degree one part $\mathbb{K}[\mathfrak{C}^d_{\mathfrak{g}}]_1$ is $\mathbb{K}$-linearly spanned by $x(k)_{ij}, \ 1\leq k\leq d, \ 1\leq i,j\leq n$. From now on, we view $X(k)$ $(1\leq k\leq d)$ as matrices over $\mathbb{K}[\mathfrak{C}^d_{\mathfrak{g}}]$. Note these matrices are mutually commutative by the definition of $I$, and $(X(1),\cdots, X(d))\in \mathfrak{C}^d_{\mathfrak{g}}(\mathbb{K}[\mathfrak{C}^d_{\mathfrak{g}}])$ can also be viewed as the tautological $\mathbb{K}[\mathfrak{C}^d_{\mathfrak{g}}]$-valued  point of the commuting scheme $\mathfrak{C}^d_{\mathfrak{g}}$.

 Since obviously $I$ is $G$-invariant, we have the induced action of $G$ on $\mathbb{K}[\mathfrak{C}^d_{\mathfrak{g}}]$. Moreover, the action preserves the degrees on $\mathbb{K}[\mathfrak{C}^d_{\mathfrak{g}}]$, so that the invariant subring $\mathbb{K}[\mathfrak{C}^d_{\mathfrak{g}}]^G$ is still a graded $\mathbb{K}$-algebra whose degree zero part is equal to $\mathbb{K}$. Any polynomial function on $\mathfrak{g}^d$ restricts to a polynomial function on $\mathfrak{t}^d$ through the inclusion $\mathfrak{t}^d\subset \mathfrak{g}^d$, and the restriction homomorphism $\mathbb{K}[\mathfrak{g}^d]\rightarrow \mathbb{K}[\mathfrak{t}^d]$ factors through $\mathbb{K}[\mathfrak{C}^d_{\mathfrak{g}}]$. This induces  the following  restriction homomorphism between the invariant rings:
\begin{displaymath}
\Phi: \mathbb{K}[\mathfrak{C}^d_{\mathfrak{g}}]^G \rightarrow \mathbb{K}[\mathfrak{t}^d]^W.
\end{displaymath}

For $1\leq k\leq d$, let $Y(k)$ be the image of $X(k)$  under the restriction homomorphism $M_n(\mathbb{K}[\mathfrak{C}^d_{\mathfrak{g}}])\rightarrow M_n(\mathbb{K}[\mathfrak{t}^d])$. Then $Y(1),\cdots, Y(d)$ are commuting $n\times n$ matrices over $\mathbb{K}[\mathfrak{t}^d]$. 
We define the following formal power series:
\begin{equation}\label{equ:defn of F_g and   F_t}
\begin{split}
F_{\mathfrak{g}}&:=\det (I_n+\sum\limits_{(i_1,\cdots, i_d)\in S}X(1)^{i_1}X(2)^{i_2}\cdots X(d)^{i_d}T_{i_1\cdots i_d}),\\ 
F_{\mathfrak{t}}&:=\det (I_n+\sum\limits_{(i_1,\cdots, i_d)\in S}Y(1)^{i_1}Y(2)^{i_2}\cdots Y(d)^{i_d}T_{i_1\cdots i_d}),
\end{split}
\end{equation}
with $F_{\mathfrak{g}}\in \mathbb{K}[\mathfrak{C}^d_{\mathfrak{g}}][[T_{i_1\cdots i_d}|(i_1,\cdots, i_d)\in S]]$ and  $F_{\mathfrak{t}}\in \mathbb{K}[\mathfrak{t}^d][[T_{i_1\cdots i_d}|(i_1,\cdots, i_d)\in S]]$.
Here and throughout this paper,  $S$ denotes the index set  
\[\{(i_1, \cdots, i_d)\:|\: i_1, \cdots, i_d \in \mathbb{Z}_{\geq 0}, (i_1, \cdots, i_d)\neq (0,\cdots, 0), i_1+\cdots + i_d \textmd{ is even}\}.\] 
Since determinants are invariant under conjugations, we see in fact $F_{\mathfrak{g}}\in \mathbb{K}[\mathfrak{C}^d_{\mathfrak{g}}]^G [[T_{i_1\cdots i_d}|(i_1,\cdots, i_d)\in S]]$ and $F_{\mathfrak{t}}\in \mathbb{K}[\mathfrak{t}^d]^W [[T_{i_1\cdots i_d}|(i_1,\cdots, i_d)\in S]]$. Moreover, under the restriction homomorphism, we have 
$\Phi (F_{\mathfrak{g}})=F_{\mathfrak{t}}$. 

By Lemma \ref{lemma:f^2 determines f}, there exists a unique $\sqrt{F_{\mathfrak{g}}}\in \mathbb{K}[\mathfrak{C}^d_{\mathfrak{g}}]^G [[T_{i_1\cdots i_d}|(i_1,\cdots, i_d)\in S]]$ with constant coefficient  $1$ and satisfying   $\sqrt{F_{\mathfrak{g}}}^2=F_{\mathfrak{g}}$. Similarly, we denote $\sqrt{F_{\mathfrak{t}}}\in \mathbb{K}[\mathfrak{t}^d]^W [[T_{i_1\cdots i_d}|(i_1,\cdots, i_d)\in S]]$ as the unique square root of $F_{\mathfrak{t}}$ with constant coefficient $1$. Define $\mathcal{M}$ to be the set of non-empty monomials in the variables $\{T_{i_1 \cdots i_d}| (i_1, \cdots, i_d)\in S\}$. We expand these two formal power series as follows:
\begin{equation}\label{equ:expansion F_t, F_g}
 \sqrt{F_{\mathfrak{g}}}=1+\sum\limits_{w\in \mathcal{M}}d_w w,\ \ \ \sqrt{F_{\mathfrak{t}}}= 1+\sum\limits_{w\in \mathcal{M}} c_w w,
\end{equation} 
with  $d_w\in \mathbb{K}[\mathfrak{C}^d_{\mathfrak{g}}]^G$,\ $c_w\in \mathbb{K}[\mathfrak{t}^d]^W$.

For a commutative ring $R$, 
and for a nonzero formal power series $f\in R[[T_{i_1 \cdots i_d}| (i_1, \cdots, i_d)\in S]] $, we write $f=a_0+\sum\limits_{w\in \mathcal{M}}a_w w$ with $a_0,\ a_w\in R$. Define the degree $\deg f\in \mathbb{Z}_{\geq 0}\cup \{+\infty\}$ of $f$ by 
\[\deg f:=
\begin{cases}
0, \ \textmd{if } f=a_0;\\ 
\sup \{\deg w| w\in \mathcal{M}, a_w\neq 0\}, \ \textmd{if } f\neq a_0.
\end{cases}
\]
We say $f$ is \textit{of finite degree} if $\deg f<+\infty$.

Let $\bigoplus \limits_S\mathbb{K}$ be the direct sum of $\mathbb{K}$ indexed by $S$. Note that for an element $\underline{t}=(t_{i_1\cdots i_d})\in \bigoplus \limits_S\mathbb{K} $, the components  $t_{i_1 \cdots i_d}=0$ for all but finitely many $(i_1, \cdots,  i_d)\in S$. Then  if a formal power series $f\in R[[T_{i_1 \cdots i_d}| (i_1, \cdots, i_d)\in S]] $ is of finite degree, we get a well-defined element $f(\underline{t})\in R$ by evaluating $f$ at $T_{i_1 \cdots i_d}=t_{i_1 \cdots i_d}$. Note $f(\underline{0})\in R$ is the constant coefficient of $f$, where $\underline{0}$ is the zero element in $\bigoplus \limits_S\mathbb{K}$.

Let $G$ act on $M_n(\mathbb{K})^d$ via simultaneous conjugation: $g\cdot(A_1,\cdots, A_d)=(gA_1 g^{-1},\cdots, gA_d g^{-1})$, \ for $g\in G$, \ $(A_1,\cdots, A_d)\in M_n(\mathbb{K})^d$. Then the inclusion $\mathfrak{g}^d\subset M_n(\mathbb{K})^d$ induces a restriction homomorphism $\varphi: \mathbb{K}[M_n(\mathbb{K})^d]^G\rightarrow \mathbb{K}[\mathfrak{C}^d_{\mathfrak{g}}]^G$, and we have the following diagram:
\begin{displaymath}
\mathbb{K}[M_n(\mathbb{K})^d]^G \xrightarrow{\varphi} 
\mathbb{K}[\mathfrak{C}^d_{\mathfrak{g}}]^G \xrightarrow{\Phi} \mathbb{K}[\mathfrak{t}^d]^W.
\end{displaymath}


\begin{proposition}\label{prop:image varphi=V}
 The image of $\varphi$ is $\mathbb{K}$-linearly spanned by the coefficients $\{1\}\cup\{d_w| w\in \mathcal{M}\}$ of $\sqrt{F_{\mathfrak{g}}}$.
\end{proposition}

\begin{proof}
Let $S^{\prime}:=\{(i_1, \cdots, i_d)| i_1, \cdots, i_d\in \mathbb{Z}_{\geq 0}, i_1+\cdots +i_d \textmd{ odd}\} $, and define the following $n\times n$ matrix
 \[A:=I_n+ \sum_{(i_1,\cdots, i_d)\in S\cup S^{\prime}}X(1)^{i_1}\cdots X(d)^{i_d}T_{i_1\cdots i_d}.\]
Let $F ^{\prime}:=\det A$, a formal power series of the variables $\{T_{i_1\cdots i_d}|(i_1,\cdots, i_d)\in S\cup S ^{\prime}\}$. 
By Procesi [\ref{procesi:matrix invariants}] (char.  $0$ case) and Zubkov [\ref{zubkov}]  (char.  $p>2$ case),   $\mathrm{Im}\ \varphi$ is generated as a $\mathbb{K}$-algebra by the coefficients of $\mathrm{tr }\ A$. From this or directly  from Lopatin [\ref{lopatin}, Corollary 2], we deduce that  $\mathrm{Im}\ \varphi$ is generated by the coefficients of   $F ^{\prime}$, and hence by the coefficients of $(F ^{\prime})^2$ according to Lemma \ref{lemma:f^2 determines f} (2). Note if $G=O_n(\mathbb{K})$ or $SO_n(\mathbb{K})$, the matrix $X(1)^{i_1}\cdots X(d)^{i_d}$ is symmetric for $(i_1,\cdots, i_d)\in S$ and  skew-symmetric for $(i_1,\cdots, i_d)\in S ^{\prime}$, so in these cases
\begin{eqnarray*}
(F ^{\prime})^2&&=\det (A  A^t)\\
&&=\det [(I_n+ \sum_{(i_1,\cdots, i_d)\in S}X(1)^{i_1}\cdots X(d)^{i_d}T_{i_1\cdots i_d})^2-\\ 
&&
(\sum_{(i_1,\cdots, i_d)\in S ^{\prime}}X(1)^{i_1}\cdots X(d)^{i_d}T_{i_1\cdots i_d})^2].
\end{eqnarray*}

Similarly if $G=Sp_n(\mathbb{K})$, we have 
\begin{eqnarray*}
(F ^{\prime})^2 &&=\det (A  (JAJ)^t) \\ 
&&=\det [(I_n+ \sum_{(i_1,\cdots, i_d)\in S}X(1)^{i_1}\cdots X(d)^{i_d}T_{i_1\cdots i_d})^2-\\ 
&&(\sum_{(i_1,\cdots, i_d)\in S ^{\prime}}X(1)^{i_1}\cdots X(d)^{i_d}T_{i_1\cdots i_d})^2].
\end{eqnarray*}
In either case, each coefficient of $(F ^{\prime})^2$ is a $\mathbb{K}$-linear combination of the coefficients of  $F_{\mathfrak{g}}=\det (I_n+ \sum_{(i_1,\cdots, i_d)\in S}X(1)^{i_1}\cdots X(d)^{i_d}T_{i_1\cdots i_d})$. Then $\mathrm{Im}\ \varphi$ is generated by the coefficients of $F_{\mathfrak{g}}$, or equivalently, by the evaluations $F_{\mathfrak{g}}(\underline{t})$ $(\underline{t}\in \bigoplus\limits_{S}\mathbb{K})$. By the multiplicative property of determinants, for any $\underline{t}_1, \underline{t}_2\in \bigoplus\limits_{S}\mathbb{K}$, there exists $\underline{t}_3\in \bigoplus\limits_{S}\mathbb{K}$, such that 
\begin{equation}\label{equ:multi phi}
F_{\mathfrak{g}}(\underline{t}_1) F_{\mathfrak{g}}(\underline{t}_2)=F_{\mathfrak{g}}(\underline{t}_3). 
\end{equation}
So $\mathrm{Im}\ \varphi$ is $\mathbb{K}$-linearly spanned by $F_{\mathfrak{g}}(\underline{t})$ $(\underline{t}\in \bigoplus\limits_{S}\mathbb{K})$.

By Lemma \ref{lemma:f^2 determines f} (2), the $\mathbb{K}$-subalgebra of $\mathbb{K}[\mathfrak{C}^d_{\mathfrak{g}}]^G$ generated   by the coefficients of $\sqrt{F_{\mathfrak{g}}}$ coincides with that generated by the coefficients of  $F_{\mathfrak{g}}$, and hence is equal to $\mathrm{Im}\ \varphi$. Moreover, the equality (\ref{equ:multi phi}) implies that $(\sqrt{F_{\mathfrak{g}}}(\underline{t}_1) \sqrt{F_{\mathfrak{g}}}(\underline{t}_2))^2=(\sqrt{F_{\mathfrak{g}}}(\underline{t}_3))^2$. By applying Lemma \ref{lemma:f^2 determines f, first form} to the graded $\mathbb{K}$-algebra $\mathbb{K}[\mathfrak{C}^d_{\mathfrak{g}}]^G$, we obtain $\sqrt{F_{\mathfrak{g}}}(\underline{t}_1) \sqrt{F_{\mathfrak{g}}}(\underline{t}_2)=\sqrt{F_{\mathfrak{g}}}(\underline{t}_3)$. So finally $\mathrm{Im}\ \varphi$ is $\mathbb{K}$-linearly spanned by the coefficients of $\sqrt{F_{\mathfrak{g}}}$.
\end{proof}

The following two lemmas are important,  but technical in nature, so their proofs will be postponed to Section \ref{technical lemmas}. Here and throughout this paper, $\mathcal{M}(\lfloor\frac{n}{2}\rfloor):=\{w\in \mathcal{M}| \deg w\leq \lfloor\frac{n}{2}\rfloor\}$ is the subset of $\mathcal{M}$ consisting of monomials with degree less than or equal to $\lfloor\frac{n}{2}\rfloor$.

\begin{lemma}\label{lemma:deg of sqrt F_t}
 $\deg \sqrt{F_{\mathfrak{t}}}\leq \lfloor\frac{n}{2}\rfloor$, and $\{1\}\cup \{c_w| w\in \mathcal{M}(\lfloor\frac{n}{2}\rfloor)\}$ is a $\mathbb{K}$-linear basis of $\mathbb{K}[\mathfrak{t}^d]^W$. 
\end{lemma}
\begin{lemma}\label{lemma:deg of sqrt F_g}
 $\deg \sqrt{F_{\mathfrak{g}}}\leq \lfloor\frac{n}{2}\rfloor$. 
\end{lemma}
Assuming these lemmas, we can now prove one of  our main theorems.

\begin{theorem}\label{main thm: char.  0}
If $\rm{char} \ \mathbb{K}=0$, then the restriction homomorphism $\Phi: \mathbb{K}[\mathfrak{C}^d_{\mathfrak{g}}]^G\rightarrow \mathbb{K}[\mathfrak{t}^d]^W$ is an isomorphism of $\mathbb{K}$-algebras. 
\end{theorem}
\begin{proof}

By Lemma \ref{lemma:deg of sqrt F_t} and \ref{lemma:deg of sqrt F_g}, the expansions (\ref{equ:expansion F_t, F_g}) reduce to the following:
\begin{equation}\label{equ:final expansion F_t, F_g}
  \sqrt{F_{\mathfrak{g}}}=1+\sum\limits_{w\in \mathcal{M}(\lfloor\frac{n}{2}\rfloor)}d_w w, \ \ \  \sqrt{F_{\mathfrak{t}}}= 1+\sum\limits_{w\in \mathcal{M}(\lfloor\frac{n}{2}\rfloor)} c_w w .
\end{equation} 
Since  $\Phi(\sqrt{F_{\mathfrak{g}}})^2=\Phi(F_{\mathfrak{g}})=F_{\mathfrak{t}}=(\sqrt{F_{\mathfrak{t}}})^2$, and both the constant coefficients of $\Phi(\sqrt{F_{\mathfrak{g}}})$, \ $\sqrt{F_{\mathfrak{t}}}$ are  $1$, we deduce from Lemma \ref{lemma:f^2 determines f} that $\Phi(\sqrt{F_{\mathfrak{g}}})=\sqrt{F_{\mathfrak{t}}}$. So 
\begin{equation}\label{equ:d_w, c_w}
\Phi(d_w)=c_w, \quad \text{ for all } w\in \mathcal{M}(\lfloor\frac{n}{2}\rfloor). 
\end{equation}

Since   $\rm{char}\ \mathbb{K}=0$, the reductive group $G$ is linearly  reductive. Then $\varphi: \mathbb{K}[M_n(\mathbb{K})^d]^G \rightarrow\mathbb{K}[\mathfrak{C}^d_{\mathfrak{g}}]^G$ is surjective, because it is induced from  the surjective restriction homomorphism $\mathbb{K}[M_n(\mathbb{K})^d]\twoheadrightarrow \mathbb{K}[\mathfrak{C}^d_{\mathfrak{g}}]$. 
By Proposition \ref{prop:image varphi=V} and \eqref{equ:final expansion F_t, F_g}, $\mathbb{K}[\mathfrak{C}^d_{\mathfrak{g}}]^G$ is $\mathbb{K}$-linearly spanned by $\{1\}\cup\{d_w| w\in \mathcal{M}(\lfloor\frac{n}{2}\rfloor)\}$. Then by Lemma \ref{lemma:deg of sqrt F_t} and \eqref{equ:d_w, c_w},  the homomorphism $\Phi$  maps a $\mathbb{K}$-linearly spanning set bijectively to a $\mathbb{K}$-linear basis, hence it is an isomorphism.

\end{proof}

Let $\mathbb{K}[\mathfrak{C}^d_{\mathfrak{g},red}]:= \mathbb{K}[\mathfrak{C}^d_{\mathfrak{g}}]/\sqrt{(0)}$ be the quotient of $\mathbb{K}[\mathfrak{C}^d_{\mathfrak{g}}]$ by its nilpotent radical. This is the coordinate ring of the commuting variety $\mathfrak{C}^d_{\mathfrak{g},red}$.  As $\mathbb{K}[\mathfrak{t}^d]^W$ is reduced, $\Phi: \mathbb{K}[\mathfrak{C}^d_{\mathfrak{g}}]^G\rightarrow \mathbb{K}[\mathfrak{t}^d]^W$ factors through $\mathbb{K}[\mathfrak{C}^d_{\mathfrak{g},red}]^G$. 

\begin{theorem}\label{thm:char.  p case}
If $\rm{char} \ \mathbb{K}=p>2$, then the restriction homomorphism $\Phi: \mathbb{K}[\mathfrak{C}^d_{\mathfrak{g},red}]^G\rightarrow \mathbb{K}[\mathfrak{t}^d]^W$ is an isomorphism of $\mathbb{K}$-algebras. 
\end{theorem}
\begin{proof}

Let $\pi: \mathbb{K}[\mathfrak{C}^d_{\mathfrak{g}}]^G\rightarrow \mathbb{K}[\mathfrak{C}^d_{\mathfrak{g},red}]^G$ be the homomorphism induced by the natural quotient homomorphism $\mathbb{K}[\mathfrak{C}^d_{\mathfrak{g}}]\rightarrow \mathbb{K}[\mathfrak{C}^d_{\mathfrak{g},red}]$. Consider the composition of homomorphisms:
\begin{displaymath}
\mathbb{K}[M_n(\mathbb{K})^d]^G\xrightarrow{\varphi}\mathbb{K}[\mathfrak{C}^d_{\mathfrak{g}}]^G\xrightarrow{\pi} \mathbb{K}[\mathfrak{C}^d_{\mathfrak{g},red}]^G. 
\end{displaymath}

By the same arguments as above, we can see that  when restricted on the subspace $\mathrm{ Im }\ (\pi\circ \varphi)$,  $\Phi$ induces an isomorphism $\mathrm{ Im }\ (\pi\circ \varphi)\xrightarrow{\sim} \mathbb{K}[\mathfrak{t}^d]^W$ of $\mathbb{K}$-linear spaces. In particular, $\Phi: \mathbb{K}[\mathfrak{C}^d_{\mathfrak{g},red}]^G\rightarrow \mathbb{K}[\mathfrak{t}^d]^W$ is surjective. 

For any $f\in \mathbb{K}[\mathfrak{C}^d_{\mathfrak{g},red}]^G$, it follows from Mumford-Fogarty-Kirwan [\ref{GIT}, Lemma A.1.2] that  there exists $m\geq 1$, such that $f^m\in \mathrm{ Im }~ (\pi\circ \varphi)$. If $\Phi(f)=0$, then $\Phi(f^m)=0$. Since $f^m \in \mathrm{ Im }~ (\pi\circ \varphi)$ and $\Phi|_{\mathrm{ Im }~ (\pi\circ \varphi)}$ is an isomorphism, we obtain $f^m=0$. This implies $f=0$ as the ring $\mathbb{K}[\mathfrak{C}^d_{\mathfrak{g},red}]^G$ is reduced. So we obtain $\Phi$ is injective. This in turn implies that $\Phi: \mathbb{K}[\mathfrak{C}^d_{\mathfrak{g},red}]^G\rightarrow \mathbb{K}[\mathfrak{t}^d]^W$ is an isomorphism.

\end{proof}

\section{Main theorems: type $D$ case} \label{sec:The n even and G=SO_n case}
In this section, keeping the same notations as in Section \ref{sec:main theorems}, we assume furthermore that $n\geq 2$ is even and  $G=O_n(\mathbb{K})$.  Let   $G^{\prime}:=SO_n(\mathbb{K})$, and define the corresponding  Weyl group by $W^{\prime}:=N_{G^{\prime}}(T)/T$.  Note both $G^{\prime}\subset G$ and $W^{\prime}\subset W=N_{G}(T)/T$ are subgroups of index two. We have the following commutative diagram.
\begin{displaymath}
\begin{diagram}
\mathbb{K}[\mathfrak{C}^d_{\mathfrak{g}}]^{G^{\prime}} &\rTo^{\ \ \ \ \Phi \ \ \ \ } &\mathbb{K}[\mathfrak{t}^d]^{W^{\prime}}\\
\uInto_{} & &\uInto_{}\\
\mathbb{K}[\mathfrak{C}^d_{\mathfrak{g}}]^G & \rTo^{\ \ \ \ \Phi \ \ \ \ }&\mathbb{K}[\mathfrak{t}^d]^W
\end{diagram}
\end{displaymath}

Take $w_0\in W$ which generates the quotient group $W/W^{\prime}\simeq \mathbb{Z}/2\mathbb{Z}$. Note $W/W^{\prime}$ acts naturally on $\mathbb{K}[\mathfrak{t}^d]^{W^{\prime}}$ and we have the eigen-subspace decomposition
\[\mathbb{K}[\mathfrak{t}^d]^{W^{\prime}}=\mathbb{K}[\mathfrak{t}^d]^{W^{\prime}}_{(0)}\oplus \mathbb{K}[\mathfrak{t}^d]^{W^{\prime}}_{(1)},\]
where  $\mathbb{K}[\mathfrak{t}^d]^{W^{\prime}}_{(0)}=\mathbb{K}[\mathfrak{t}^d]^W$ is the invariant part and 
\[\mathbb{K}[\mathfrak{t}^d]^{W^{\prime}}_{(1)}=\{v\in \mathbb{K}[\mathfrak{t}^d]^{W^{\prime}}|w_0 v=-v\}.\]
 Similarly, the action of $G/G^{\prime}\simeq \mathbb{Z}/2\mathbb{Z} $ on $\mathbb{K}[\mathfrak{C}^d_{\mathfrak{g}}]^{G^{\prime}}$ induces the eigen-subspace decomposition $\mathbb{K}[\mathfrak{C}^d_{\mathfrak{g}}]^{G^{\prime}}=\mathbb{K}[\mathfrak{C}^d_{\mathfrak{g}}]^{G^{\prime}}_{(0)}\oplus \mathbb{K}[\mathfrak{C}^d_{\mathfrak{g}}]^{G^{\prime}}_{(1)}$, with  $\mathbb{K}[\mathfrak{C}^d_{\mathfrak{g}}]^{G^{\prime}}_{(0)}=\mathbb{K}[\mathfrak{C}^d_{\mathfrak{g}}]^G$  and $\mathbb{K}[\mathfrak{C}^d_{\mathfrak{g}}]^{G^{\prime}}_{(1)}=\{v\in \mathbb{K}[\mathfrak{C}^d_{\mathfrak{g}}]^{G^{\prime}}|g_0 v=-v\}$, where $g_0$ is a generator of $G/G ^{\prime}$. Clearly the restriction homomorphism $\Phi$ preserves these decompositions: $\Phi(\mathbb{K}[\mathfrak{C}^d_{\mathfrak{g}}]^{G^{\prime}}_{(i)})\subset \mathbb{K}[\mathfrak{t}^d]^{W^{\prime}}_{(i)}$, \ $i=0,1$.

Let $S^{\prime}:=\{(i_1, \cdots, i_d)| i_1, \cdots, i_d\in \mathbb{Z}_{\geq 0}, i_1+\cdots +i_d \textmd{ odd}\} $,  and $\mathcal{M}^{\prime}(\frac{n}{2})$ be the set of non-empty monomials in the variables $\{T_{i_1 \cdots i_d}| (i_1, \cdots, i_d)\in S^{\prime}\}$ whose degree is less than or equal to $\frac{n}{2}$.
Define the following formal power series:
\begin{displaymath}
\begin{split}
H_{\mathfrak{g}}&:={\rm Pf} (\sum\limits_{(i_1,\cdots, i_d)\in S^{\prime }}X(1)^{i_1}X(2)^{i_2}\cdots X(d)^{i_d}T_{i_1\cdots i_d}),\\ 
H_{\mathfrak{t}}&:={\rm Pf} (\sum\limits_{(i_1,\cdots, i_d)\in S^{\prime }}Y(1)^{i_1}Y(2)^{i_2}\cdots Y(d)^{i_d}T_{i_1\cdots i_d}),
\end{split}
\end{displaymath}
with $H_{\mathfrak{g}}\in \mathbb{K}[\mathfrak{C}^d_{\mathfrak{g}}][[T_{i_1\cdots i_d}|(i_1,\cdots, i_d)\in S^{\prime}]]$,  $H_{\mathfrak{t}}\in \mathbb{K}[\mathfrak{t}^d][[T_{i_1\cdots i_d}|(i_1,\cdots, i_d)\in S^{\prime}]]$.
Here ${\rm Pf}(A)$ means the Pfaffian of  a skew symmetric matrix $A$. 
Note for any $P\in G=O_n(\mathbb{K})$, 
\[P\cdot H_{\mathfrak{g}}={\rm Pf} (\sum\limits_{(i_1,\cdots, i_d)\in S^{\prime }}P^{-1 } X(1)^{i_1}\cdots X(d)^{i_d} P T_{i_1\cdots i_d})=\det P \cdot  H_{\mathfrak{g}},\]
 so $H_{\mathfrak{g}}\in \mathbb{K}[\mathfrak{C}^d_{\mathfrak{g}}]_{(1)}^{G^{\prime}}[[T_{i_1\cdots i_d}|(i_1,\cdots, i_d)\in S]]$. In a similar way, we can see that  $H_{\mathfrak{t}}\in \mathbb{K}[\mathfrak{t}^d]_{(1)}^{W^{\prime}}[[T_{i_1\cdots i_d}|(i_1,\cdots, i_d)\in S^{\prime}]]$.
By  constructions, $\deg H_{\mathfrak{g}}\leq \frac{n}{2}$ and  $\deg H_{\mathfrak{t}}\leq \frac{n}{2}$, so we can write
\begin{equation}\label{equ:expansion H_t, H_g}
  H_{\mathfrak{g}}=\sum\limits_{w\in \mathcal{M}^{\prime}(\frac{n}{2})}d^{\prime }_w w,  \ \  \ H_{\mathfrak{t}}= \sum\limits_{w\in \mathcal{M}^{\prime}(\frac{n}{2})} c^{\prime}_w w, 
\end{equation} 
with    $d ^{\prime}_w\in \mathbb{K}[\mathfrak{C}^d_{\mathfrak{g}}]_{(1)}^{G ^{\prime}}$, \ $c ^{\prime}_w\in \mathbb{K}[\mathfrak{t}^d]_{(1)}^{W ^{\prime}}$.

We postpone the proof of  the next lemma to Section \ref{technical lemmas}.
\begin{lemma}\label{lemma:basis of K[t]^W_(1)}
The set $\{c ^{\prime}_w | w\in \mathcal{M} ^{\prime}(\frac{n}{2})\}$ is a $\mathbb{K}$-linear basis of $\mathbb{K}[\mathfrak{t}^d]^{W^{\prime}}_{(1)}$.
\end{lemma}

Now consider the composition of homomorphisms:
\begin{displaymath}
\mathbb{K}[M_n(\mathbb{K})^d]^{G ^{\prime}}\xrightarrow{ \ \varphi \ } \mathbb{K}[\mathfrak{C}^d_{\mathfrak{g}}]^{G^{\prime}} \xrightarrow{\ \pi_1 \ } \mathbb{K}[\mathfrak{C}^d_{\mathfrak{g}}]^{G^{\prime}}_{(1)},
\end{displaymath}
where $\varphi$ is the restriction homomorphism and $\pi_1$ is the projection under the decomposition $\mathbb{K}[\mathfrak{C}^d_{\mathfrak{g}}]^{G^{\prime}}=\mathbb{K}[\mathfrak{C}^d_{\mathfrak{g}}]^{G^{\prime}}_{(0)}\oplus \mathbb{K}[\mathfrak{C}^d_{\mathfrak{g}}]^{G^{\prime}}_{(1)} $. 

\begin{proposition}\label{prop:Im=V_(1)}
The image of $\pi_1\circ \varphi$ is $\mathbb{K}$-linearly spanned by $\{d ^{\prime}_{w}| w\in \mathcal{M} ^{\prime}(\frac{n}{2})\}$.
\end{proposition}
\begin{proof}
Note that $\mathrm{Im }\ \varphi=\mathrm{Im }\ \varphi\cap \mathbb{K}[\mathfrak{C}^d_{\mathfrak{g}}]^{G^{\prime}}_{(0)}\oplus \mathrm{Im }\ \varphi\cap \mathbb{K}[\mathfrak{C}^d_{\mathfrak{g}}]^{G^{\prime}}_{(1)}$, since $\varphi$ is $G/G^{\prime}$-equivariant. 
We can deduce from  Lopatin [\ref{lopatin}, Corollary 2] that    $\mathrm{ Im}\  \varphi $ is generated as a $\mathbb{K}$-algebra by $d ^{\prime}_{w} \ (w\in \mathcal{M} ^{\prime}(\frac{n}{2}))$ and $\mathrm{Im }\ \varphi\cap \mathbb{K}[\mathfrak{C}^d_{\mathfrak{g}}]^{G^{\prime}}_{(0)}$.  Note also  $\mathbb{K}[\mathfrak{C}^d_{\mathfrak{g}}]^{G^{\prime}}_{(i)}\cdot \mathbb{K}[\mathfrak{C}^d_{\mathfrak{g}}]^{G^{\prime}}_{(j)}\subset \mathbb{K}[\mathfrak{C}^d_{\mathfrak{g}}]^{G^{\prime}}_{(i+j)}$ for $i,j \in \mathbb{Z}/2\mathbb{Z}=\{0,1\}$. It follows that $\mathrm{ Im}\  \varphi\cap \mathbb{K}[\mathfrak{C}^d_{\mathfrak{g}}]^{G^{\prime}}_{(1)} $ is $\mathbb{K}$-linearly spanned by 
\[\{a\cdot d ^{\prime}_{w}| a\in \mathrm{Im }\ \varphi\cap \mathbb{K}[\mathfrak{C}^d_{\mathfrak{g}}]^{G^{\prime}}_{(0)}, \ w\in \mathcal{M} ^{\prime}(\frac{n}{2})\}.\]

By the proof of Proposition \ref{prop:image varphi=V}, $\mathrm{Im} \ \varphi\cap \mathbb{K}[\mathfrak{C}^d_{\mathfrak{g}}]^{G^{\prime}}_{(0)}=\varphi(\mathbb{K}[M_n(\mathbb{K})^d]^G)$ is $\mathbb{K}$-linearly spanned by the evaluations $F_{\mathfrak{g}}(\underline{t})$ $(\underline{t}\in \bigoplus\limits_{S}\mathbb{K})$. On the other hand, it is easy to see the $\mathbb{K}$-linear subspace spanned by the coefficients $\{d ^{\prime}_{w}|  w\in \mathcal{M} ^{\prime}(\frac{n}{2})\}$ coincides with the $\mathbb{K}$-linear subspace spanned by the evaluations $H_{\mathfrak{g}}(\underline{t})$ $(\underline{t}\in \bigoplus\limits_{S^{\prime}}\mathbb{K })$. It follows then that $\mathrm{ Im}\  \varphi\cap \mathbb{K}[\mathfrak{C}^d_{\mathfrak{g}}]^{G^{\prime}}_{(1)} $ is $\mathbb{K}$-linearly spanned by $\{F_{\mathfrak{g}}(\underline{t}_1)\cdot H_{\mathfrak{g}}(\underline{t}_2)| \underline{t}_1\in \bigoplus\limits_{S}\mathbb{K}, \ \underline{t}_2\in \bigoplus\limits_{S ^{\prime}}\mathbb{K} \}$.

As $\det A \cdot {\rm Pf}(B)={\rm Pf}(A B A)$ for any $n\times n$ symmetric matrix $A$ and skew symmetric matrix $B$, we can see for any $\underline{t}_1\in \bigoplus\limits_{S}\mathbb{K}$ and any  $ \underline{t}_2\in \bigoplus\limits_{S ^{\prime}}\mathbb{K}$, there exists $\underline{t}_3\in \bigoplus\limits_{S ^{\prime}}\mathbb{K}$, such that $F_{\mathfrak{g}}(\underline{t}_1)\cdot H_{\mathfrak{g}}(\underline{t}_2)=H_{\mathfrak{g}}(\underline{t}_3)$.
So $\mathrm{ Im } \ \pi_1\circ \varphi=\mathrm{ Im}\  \varphi\cap \mathbb{K}[\mathfrak{C}^d_{\mathfrak{g}}]^{G^{\prime}}_{(1)} $ is $\mathbb{K}$-linearly spanned by $\{H_{\mathfrak{g}}(\underline{t})|  \underline{t}\in \bigoplus\limits_{S ^{\prime}}\mathbb{K} \}$, and hence by the coefficients $\{d ^{\prime}_{w}| w\in \mathcal{M} ^{\prime}(\frac{n}{2})\}$.
 \end{proof}
\begin{theorem}\label{thm:main thm in the n even G=SO_n case}
If $\rm{char}\  \mathbb{K} =0$, then the restriction  homomorphism $\Phi:\mathbb{K}[\mathfrak{C}^d_{\mathfrak{g}}]^{G^{\prime}} \rightarrow \mathbb{K}[\mathfrak{t}^d]^{W^{\prime}}$ is an isomorphism of $\mathbb{K}$-algebras.
\end{theorem}
\begin{proof}
By Theorem \ref{main thm: char.  0}, $\Phi$ induces an isomorphism $\mathbb{K}[\mathfrak{C}^d_{\mathfrak{g}}]^{G^{\prime}}_{(0)}\xrightarrow{\ \ \sim\ \ } \mathbb{K}[\mathfrak{t}^d]^{W^{\prime}}_{(0)} $.  Since $\rm{char} \ \mathbb{K} =0$, the special orthogonal  group $G ^{\prime}$ is linearly reductive, and then $\varphi: \mathbb{K}[M_n(\mathbb{K})^d] ^{G^{\prime}}\rightarrow \mathbb{K}[\mathfrak{C}^d_{\mathfrak{g}}]^{G^{\prime}}$ is surjective as it is induced from the surjection $\mathbb{K}[M_n(\mathbb{K})^d]\rightarrow \mathbb{K}[\mathfrak{C}^d_{\mathfrak{g}}]$. So by Proposition \ref{prop:Im=V_(1)}, $\mathbb{K}[\mathfrak{C}^d_{\mathfrak{g}}]^{G^{\prime}}_{(1)}$ is $\mathbb{K}$-linearly spanned by  $\{d ^{\prime}_{w}| w\in \mathcal{M} ^{\prime}(\frac{n}{2})\}$. Note $\Phi(H_{\mathfrak{g}})=H_{\mathfrak{t}}$,  so $\Phi(d ^{\prime}_w)=c ^{\prime}_w$,  for any $ w\in \mathcal{M} ^{\prime}(\frac{n}{2})$. By Lemma \ref{lemma:basis of K[t]^W_(1)}, the set $\{c ^{\prime}_w | w\in \mathcal{M} ^{\prime}(\frac{n}{2})\}$ is a $\mathbb{K}$-linear basis of $\mathbb{K}[\mathfrak{t}^d]^{W^{\prime}}_{(1)}$. It follows that $\Phi$ maps the $\mathbb{K}$-linearly spanning set $\{d ^{\prime}_w | w\in \mathcal{M} ^{\prime}(\frac{n}{2})\}$ of $\mathbb{K}[\mathfrak{C}^d_{\mathfrak{g}}]^{G^{\prime}}_{(1)}$ bijectively to the $\mathbb{K}$-linear basis $\{c ^{\prime}_w | w\in \mathcal{M} ^{\prime}(\frac{n}{2})\}$ of $\mathbb{K}[\mathfrak{t}^d]^{W^{\prime}}_{(1)}$, and hence $\Phi$ induces a $\mathbb{K}$-linear isomorphism $\mathbb{K}[\mathfrak{C}^d_{\mathfrak{g}}]^{G^{\prime}}_{(1)}\xrightarrow{\ \ \sim \ \ } \mathbb{K}[\mathfrak{t}^d]^{W^{\prime}}_{(1)}$. Combining  with the isomorphism $\mathbb{K}[\mathfrak{C}^d_{\mathfrak{g}}]^{G^{\prime}}_{(0)}\xrightarrow{\ \ \sim \ \ } \mathbb{K}[\mathfrak{t}^d]^{W^{\prime}}_{(0)}$, we finish the proof.
\end{proof}

\begin{theorem}\label{thm:n even G=SO_n, char.  p case}
If $\rm{char}\  \mathbb{K} =p>2$, then the restriction homomorphism  
\[\Phi: \mathbb{K}[\mathfrak{C}^d_{\mathfrak{g},red}]^{G^{\prime}}\rightarrow \mathbb{K}[\mathfrak{t}^d]^{W^{\prime}}\]
 is an isomorphism of $\mathbb{K}$-algebras.
  
\end{theorem}
\begin{proof}
By Theorem \ref{thm:char.  p case}, $\Phi$ induces an isomorphism $\mathbb{K}[\mathfrak{C}^d_{\mathfrak{g},red}]^{G^{\prime}}_{(0)}\xrightarrow{\ \ \sim \ \ } \mathbb{K}[\mathfrak{t}^d]^{W^{\prime}}_{(0)}$. So it suffices to show $\mathbb{K}[\mathfrak{C}^d_{\mathfrak{g},red}]^{G^{\prime}}_{(1)}\xrightarrow{\ \ \Phi \ \ } \mathbb{K}[\mathfrak{t}^d]^{W^{\prime}}_{(1)}$ is an isomorphism.

 Let  $V$ be the image of the following composition of homomorphisms:
\begin{displaymath}
\mathbb{K}[M_n(\mathbb{K})^d]^{G ^{\prime}}\xrightarrow{\varphi}\mathbb{K}[\mathfrak{C}^d_{\mathfrak{g}}]^{G ^{\prime}} \xrightarrow{} \mathbb{K}[\mathfrak{C}^d_{\mathfrak{g},red}]^{G ^{\prime}} \rightarrow \mathbb{K}[\mathfrak{C}^d_{\mathfrak{g},red}]^{G ^{\prime}}_{(1)}. 
\end{displaymath}
By Proposition \ref{prop:Im=V_(1)},  $V$ is $\mathbb{K}$-linearly spanned by $\{d ^{\prime}_{w}| w\in \mathcal{M} ^{\prime}(\frac{n}{2})\}$. The same arguments as above show that the restriction $\Phi|_{V}: V\rightarrow \mathbb{K}[\mathfrak{t}^d]^{W ^{\prime}}$ is an isomorphism. In particular, $\Phi: \mathbb{K}[\mathfrak{C}^d_{\mathfrak{g},red}]^{G^{\prime}}_{(1)}\rightarrow \mathbb{K}[\mathfrak{t}^d]^{W^{\prime}}_{(1)}$ is surjective. 

Since $\mathbb{K}[M_n(\mathbb{K})^d]\xrightarrow{} \mathbb{K}[\mathfrak{C}^d_{\mathfrak{g},red}]$ is surjective, it follows from Mumford-Fogarty-Kirwan [\ref{GIT}, Lemma A.1.2] that  for any $f\in \mathbb{K}[\mathfrak{C}^d_{\mathfrak{g},red}]^{G ^{\prime}}$, a positive power of $f$ is in the image of $\mathbb{K}[M_n(\mathbb{K})^d]^{G ^{\prime}}$. So for any $f \in \mathbb{K}[\mathfrak{C}^d_{\mathfrak{g},red}]^{G ^{\prime}}_{(1)}$, there exists $m\geq 1$, such that $f^m\in \mathbb{K}[\mathfrak{C}^d_{\mathfrak{g},red}]^{G ^{\prime}}_{(0)}$ ($m$ even) or $f^m\in V$ ($m$ odd). If $\Phi(f)=0$, then $\Phi(f^m)=0$. Since either $\Phi: \mathbb{K}[\mathfrak{C}^d_{\mathfrak{g},red}]^{G^{\prime}}_{(0)}\xrightarrow{ } \mathbb{K}[\mathfrak{t}^d]^{W^{\prime}}_{(0)}$ or $\Phi|_{V}: V\rightarrow \mathbb{K}[\mathfrak{t}^d]^{W ^{\prime}}$ is injective, we obtain $f^m=0$, and then $f=0$ by the reducedness of $\mathbb{K}[\mathfrak{C}^d_{\mathfrak{g},red}]^{G^{\prime}}$. This shows $\Phi: \mathbb{K}[\mathfrak{C}^d_{\mathfrak{g},red}]^{G^{\prime}}_{(1)}\rightarrow \mathbb{K}[\mathfrak{t}^d]^{W^{\prime}}_{(1)}$ is injective, and we have finished the proof.
\end{proof}

\section{Proofs of some lemmas}\label{technical lemmas}
The section is the most technical part of the paper. The reader is advised to skip the section at the first reading. We keep the same notations as in Section \ref{sec:main theorems}, \ref{sec:The n even and G=SO_n case}, and let $m=\lfloor\frac{n}{2}\rfloor$. 

\subsection{Proofs of Lemma \ref{lemma:deg of sqrt F_t} and \ref{lemma:basis of K[t]^W_(1)}}\label{subsec:proof of lemma deg F_t}

Since for any two Cartan subalgebras $\mathfrak{t}_1$, $\mathfrak{t}_2$ of $\mathfrak{g}$, there exists an element $h\in G$, such that $Ad(h)\mathfrak{t}_1=\mathfrak{t}_2$ and since $Ad(h) F_{\mathfrak{g}}=F_{\mathfrak{g}}$, Lemma \ref{lemma:deg of sqrt F_t} and \ref{lemma:basis of K[t]^W_(1)} hold for $\mathfrak{t}_1$ if and only if they hold for $\mathfrak{t}_2$. As a result, we will choose a specific Cartan subalgebra $\mathfrak{t} $ for the proofs.

If $n $ is even and $G=Sp_n(\mathbb{K})$, let \[\mathfrak{t}=\{\mathrm{ diag }(x_1,\cdots, x_m, -x_1,\cdots, -x_m)\:|\: x_i\in \mathbb{K}, 1\leq i\leq m\}.\]

If $G=O_n(\mathbb{K})$ or $SO_n(\mathbb{K})$, let 
\[\mathfrak{t}=\{\textmd{SK }(x_1,\cdots, x_m) \:|\: x_i\in \mathbb{K}, 1\leq i\leq m\},\]
where $\textmd{SK }(x_1,\cdots, x_m)$ is the  $n\times n$ skew symmetric matrix:
\begin{equation}\label{equ:sk(x_1,...,x_m)}
\begin{pmatrix}
0 & \sqrt{-1}x_1 & & & \\ 
-\sqrt{-1}x_1 & 0 & & & \\ 
& & 0 & \sqrt{-1}x_2 &\\ 
& & -\sqrt{-1}x_2 & 0 & \\ 
& & & & \ddots 
\end{pmatrix}. 
\end{equation}
In other words, for $1\leq i, j\leq n$, the $(i,j)$-entry of $\textmd{SK }(x_1,\cdots, x_m)$ is
\begin{displaymath}
\begin{cases}
\sqrt{-1}x_p & \textmd{ if } (i,j)=(2p-1, 2p) \textmd{ and } 1\leq p\leq m, \\ 
-\sqrt{-1}x_p & \textmd{ if } (i,j)=(2p, 2p-1) \textmd{ and } 1\leq p\leq m, \\
0 & \textmd{ otherwise}. 
\end{cases}
\end{displaymath}

\begin{proof}[Proof of Lemma \ref{lemma:deg of sqrt F_t}]
Let $V=\mathbb{K}^m$, we identify $V$ and $\mathfrak{t}$ by:
\begin{displaymath}
  \begin{cases}
  (x_1,\cdots, x_m)\mapsto \mathrm{ diag }(x_1,\cdots, x_m, -x_1,\cdots, -x_m), \ \textmd{ if } n \textmd{ even}, G= Sp_n(\mathbb{K});\\ 
  (x_1,\cdots, x_m)\mapsto \mathrm{ SK }(x_1,\cdots, x_m), \ \textmd{ if }  G= O_n(\mathbb{K}) \textmd{ or } SO_n(\mathbb{K}).
  \end{cases}
  \end{displaymath}  
 Under this identification, the Weyl group $W$ acts on $V$ by permuting the coordinates $x_1,\cdots, x_m$ and sign changing $x_i\mapsto -x_i$.
Let $x_{ij}$ be the linear function on $V^d$ whose value at a point $(v_1,\cdots, v_d)$ is the $i$-th component of $v_j$. A direct computation shows that under the identification of $V$ and $\mathfrak{t}$, we have  $F_{\mathfrak{t}}=N_{\mathfrak{t}}^2$, where 
\begin{displaymath}
N_{\mathfrak{t}}=\prod_{k=1}^m(1+\sum\limits_{(i_1,\cdots, i_d)\in S} x_{k1}^{i_1}x_{k2}^{i_2}\cdots x_{kd}^{i_d}T_{i_1\cdots i_d}).
\end{displaymath}
The constant coefficient  of $N_{\mathfrak{t}}$ is $1$, so by the uniqueness of square root (Lemma \ref{lemma:f^2 determines f}), $N_{\mathfrak{t}}=\sqrt{F_{\mathfrak{t}}}$, and hence $\deg \sqrt{F_{\mathfrak{t}}}=\deg N_{\mathfrak{t}}\leq m=\lfloor\frac{n}{2}\rfloor$.

Next we will show that the coefficients of $N_{\mathfrak{t}}$ form a $\mathbb{K}$-linear basis of $\mathbb{K}[V^d]^W$. If $\rm{char} \  \mathbb{K} =0$, this is a direct consequence of Hunziker [\ref{Hunziker}, Lemma 2.2]. In the following we shall adapt Hunziker's proof slightly so that it works for all the $\rm{char} \ \mathbb{K} \neq 2$ cases. 

We denote by $\Lambda$ the following set of nonzero  $m\times d$ matrices:
\begin{displaymath}
\Lambda:=\{\lambda=\begin{pmatrix}
\lambda_{11} & \cdots & \lambda_{1d}\\ 
\vdots & \ddots & \vdots \\ 
\lambda_{m1} &\cdots &\lambda_{md}
\end{pmatrix}| \lambda_{ij}\in \mathbb{Z}_{\geq 0},\  \lambda \neq 0\}.
\end{displaymath}
Let $\Lambda_{even}:=\{\lambda\in \Lambda| \sum_{j=1}^d \lambda_{ij} \textmd{ is even, } \text{for} \ 1\leq i\leq m\}$.  

The symmetric group $S_m$ acts on $\Lambda$ by permuting the rows. Let $\Lambda^{+}_{even}\subset \Lambda_{even} $ be the subset of all $\lambda\in \Lambda_{even}$ such that $\lambda_1\geq \cdots \geq \lambda_m$ with respect to the lexicographic order on the rows, where $\lambda_i$ is the $i$-th row of $\lambda$. Then for $\lambda\in \Lambda_{even}$, the orbit $S_m \cdot \lambda \subset \Lambda_{even}$ contains a unique element in $\Lambda^{+}_{even}$.

To each $\lambda\in \Lambda$ corresponds to the monomials $x^{\lambda}:=\prod_{k=1}^m x_{k1}^{\lambda_{k1}}x_{k2}^{\lambda_{k2}}\cdots x_{kd}^{\lambda_{kd}}  $ and  $T_{\lambda}:=\prod_{k=1}^m T_{\lambda_{k1},\lambda_{k2},\cdots, \lambda_{kd}}$. Here by convention $T_{0,\cdots, 0}=1$.  For $\lambda\in \Lambda^{+}_{even}$ we put 
\begin{displaymath}
a_{\lambda}:=\sum\limits_{\mu\in S_m \cdot \lambda} x^{\mu}.
\end{displaymath}
Using these notations, we can rewrite $N_{\mathfrak{t}}$ as 
\begin{displaymath}
 N_{\mathfrak{t}}=1+\sum\limits_{\lambda\in \Lambda^{+}_{even}}a_{\lambda} T_{\lambda}. 
 \end{displaymath} 
It is direct to see $a_{\lambda}\in \mathbb{K}[V^d]^{W}$, and $a_{\lambda}$ \ ($\lambda\in \Lambda^{+}_{even}$) are linearly independent. Suppose $f\in \mathbb{K}[V^d]^W$ is nonzero and write $f=b_0+\sum_{\lambda\in \Lambda} b_{\lambda} x^{\lambda}$, with $b_0,\ b_{\lambda}\in \mathbb{K}$. From $w f=f$ for all sign changes $w$ we deduce $\lambda\in \Lambda_{even}$ if $b_{\lambda}\neq 0$. Then $\sum_{\lambda\in \Lambda_{even}} b_{\lambda} x^{w \lambda}=\sum_{\lambda\in \Lambda_{even}} b_{\lambda} x^{\lambda}$, \text{for} $w\in S_m$. This implies $b_{w \lambda}=b_{\lambda}$, for any $w \in S_m$. So $f$ is a linear combination of $1$ and  $a_{\lambda}$, \ $\lambda\in \Lambda^{+}_{even}$. Finally $\{1\}$ $\cup$ $\{a_{\lambda}| \lambda\in \Lambda^{+}_{even}\}$, i.e., the coefficients of $N_{\mathfrak{t}}=\sqrt{F_{\mathfrak{t}}}$, are a $\mathbb{K}$-linear basis of $\mathbb{K}[V^d]^W$. 
\end{proof}

\begin{proof}[Proof of Lemma \ref{lemma:basis of K[t]^W_(1)}]

Keep notations as above. Define 
\[\Lambda_{odd}:=\{\lambda\in \Lambda| \sum_{j=1}^d \lambda_{ij} \textmd{ is odd }, \text{for}  \ 1\leq i\leq m\},\] and let $\Lambda^{+}_{odd}\subset \Lambda_{odd} $ be the subset of all $\lambda\in \Lambda_{odd}$ such that $\lambda_1\geq \cdots \geq \lambda_m$ with respect to the lexicographic order on the rows, where $\lambda_i$ is the $i$-th row of $\lambda$. Then for $\lambda\in \Lambda_{odd}$, the orbit $S_m \cdot \lambda \subset \Lambda_{odd}$ contains a unique element in $\Lambda^{+}_{odd}$. For $\lambda\in \Lambda^{+}_{odd}$ we put 
\begin{displaymath}
a_{\lambda} =\sum\limits_{\mu\in S_m \cdot \lambda} x^{\mu}.
\end{displaymath}
Then under the identification of $V$ and $\mathfrak{t}$, a direct computation shows that   
\begin{displaymath}
 H_{\mathfrak{t}}=(\sqrt{-1})^{m}\sum\limits_{\lambda\in \Lambda^{+}_{odd}}a_{\lambda} T_{\lambda}. 
 \end{displaymath} 
 The Weyl group $W$  of $G$ acts on $V$ by permuting the coordinates $x_1,\cdots, x_d$ and sign changing $\tau_i: x_i\mapsto -x_i$. The Weyl group $W ^{\prime}$ of $G ^{\prime}$ is then the index two subgroup of $W$, generated by the permutations and $\tau_i \circ \tau_j$, \ $1\leq i, j\leq d$. It follows that 
 \begin{equation*}
 a_{\lambda}\in \mathbb{K}[V^d]^{W ^{\prime}}_{(1)}, \ \text{for all } \lambda\in \Lambda^{+}_{odd},
 \end{equation*}
where $\mathbb{K}[V^d]^{W ^{\prime}}_{(1)}$  is the $(-1)$-eigen subspace of $\mathbb{K}[V^d]^{W ^{\prime}}$ under the action of $W/W ^{\prime}$. 
 
 Moreover, it is easy to see $a_{\lambda}$ $(\lambda \in \Lambda^{+}_{odd} )$ are linearly independent. Similar to the proof of Lemma \ref{lemma:deg of sqrt F_t}, we can see $\mathbb{K}[V^d]^{W ^{\prime}}_{(1)}$ is linearly spanned by $a_{\lambda}$ $(\lambda \in \Lambda^{+}_{odd} )$. So the coefficients of $H_{\mathfrak{t}}$ form a basis of $\mathbb{K}[V^d]^{W ^{\prime}}_{(1)}$. Under the identification $V=\mathfrak{t}$, this means  $\{c ^{\prime}_w | w\in \mathcal{M} ^{\prime}(\frac{n}{2})\}$ is a $\mathbb{K}$-linear basis of $\mathbb{K}[\mathfrak{t}^d]^{W^{\prime}}_{(1)}$, as asserted.
 
\end{proof}

\subsection{Proof of Lemma \ref{lemma:deg of sqrt F_g}}\label{subsec:construction of N}

If  $n$ is even and $G=Sp_n(\mathbb{K})$, the following  matrix 
\[A:=J+J\sum_{(i_1,\cdots,i_d)\in S}X(1)^{i_1}\cdots X(d)^{i_d}T_{i_1\cdots i_d}\] 
is  skew-symmetric. According to Chen-Ng\^{o} [\ref{Chen-Ngo-Symplectic}], let 
\[N:={\rm Pf}(A) {\rm Pf}(J)^{-1}\in \mathbb{K}[\mathfrak{C}^d_{\mathfrak{g}}] [[T_{i_1\cdots i_d}|(i_1,\cdots, i_d)\in S]]. \]
Then the constant coefficient of $N$ is $1$ and 
\[N^2=\det (I_n+\sum_{(i_1,\cdots,i_d)\in S}X(1)^{i_1}\cdots X(d)^{i_d}T_{i_1\cdots i_d})=F_{\mathfrak{g}}. \]
By Lemma \ref{lemma:f^2 determines f}, we get $N=\sqrt{F_{\mathfrak{g}}}$. Since $\deg N\leq \frac{n}{2}$ by its explicit  construction, we see $\deg  \sqrt{F_{\mathfrak{g}}}\leq \frac{n}{2}$.  This gives a proof of Lemma \ref{lemma:deg of sqrt F_g} in the $G=Sp_n(\mathbb{K})$ case. 

In the remaining of this subsection, we assume  $G $ is one of the following orthogonal  groups.
\begin{displaymath}
G=
\begin{cases}
O_n(\mathbb{K}), \\  
SO_n(\mathbb{K}), & n \textmd{   odd}.
\end{cases}
\end{displaymath}

We will need the following lemma, see e.g. \cite{KSW}, \cite{Silvester}.

\begin{lemma}\label{lemma:determinant of a block matrix}
Let $R$ be a  ring, and $X_{ij}\in M_l(R)$, for $1\leq i,j\leq k$. Let $\tilde{X}\in M_{kl}(R)$ be the following block matrix
\begin{displaymath}
\begin{pmatrix}
X_{11} & X_{12} &\cdots & X_{1k}\\ 
X_{21} & X_{22} &\cdots & X_{2k}\\ 
\vdots & \vdots & \ddots &\vdots \\ 
X_{k1} & X_{k2} &\cdots  &X_{kk}
\end{pmatrix}. 
\end{displaymath}
If $X_{ij}$ commute pairwise, then 
\begin{equation*}
\det \tilde{X}=\det (\sum\limits_{\sigma \in S_k}(\mathrm{sgn }~\sigma) X_{1\sigma(1)}X_{2\sigma(2)}\cdots X_{k\sigma(k)}). 
\end{equation*}
\end{lemma}

Suppose $d\geq 1$ and consider the following $(2d+3)\times (2d+3)$ skew symmetric matrix $T$ over the polynomial ring $\mathbb{K}[t_{ij} | 1\leq i<j\leq 2d+3]$:
\begin{displaymath}
T(i,j)=
\begin{cases}
t_{ij}, \ 1\leq i<j\leq 2d+3;\\ 
-t_{ji}, \ 1\leq j<i\leq 2d+3;\\ 
0, \ 1\leq i=j\leq 2d+3.
\end{cases}
\end{displaymath}
For any $k\geq 1$ and any $1\leq i_1,\cdots, i_k\leq 2d+3$, let $T(\hat{i}_1,\cdots, \hat{i}_k)$ be the matrix obtained from $T$ by deleting the $i_j$-th row and the $i_j$-th column  for all $1\leq j\leq k$, and let $h_{i_1,\cdots, i_k}=\mathrm{Pf } \ T(\hat{i}_1,\cdots, \hat{i}_k)$ be the corresponding Pfaffian. Define the $2d \times 2d$ matrix $A$ over $\mathbb{K}[t_{ij} | 1\leq i<j\leq 2d+3]$ by $A(i,j):=h_{i,j,2d+2}^2+h_{i,2d+1,2d+2}^2+h_{j,2d+1,2d+2}^2$, for $1\leq i,j\leq 2d$.
\begin{lemma}\label{lemma:pfaffian solution}
The following system of equations for the variables $t_{ij}$ ($1\leq i<j\leq 2d+3$) has a solution in $\mathbb{K}$:
\begin{displaymath}
 \begin{cases}
 h_{2d+2}=0;\\ 
 h_{2d+3} \det A\neq 0.
 \end{cases}
 \end{displaymath} 
\end{lemma}
\begin{proof}
It suffices to show that in the polynomial ring $\mathbb{K}[t_{ij}|1\leq i<j\leq 2d+3]$, the polynomial $h_{2d+3} \det A$ is not contained in the ideal $\sqrt{(h_{2d+2})}$.   
Since the Pfaffians $h_{2d+2}$ and $h_{2d+3}$ are coprime irreducible polynomials (cf. Goodman-Wallach \cite[Lemma B.2.10]{Goodman-Wallach}), we only need to prove $\det A\notin (h_{2d+2})$.  We will proceed by induction on $d$.

 If $d=1$, by direct computations,  we have 
\[h_{2d+2}=h_4=t_{12} t_{35}-t_{13}t_{25}+t_{15}t_{23},\]
and 
\begin{eqnarray*}
\det A&=&4h_{1,3,4}^2h_{2,3,4}^2-(h_{1,2,4}^2+h_{1,3,4}^2+h_{2,3,4}^2)^2\\
&=&-(t_{35}^2+(t_{25}+t_{15})^2)(t_{35}^2+(t_{25}-t_{15})^2).
 \end{eqnarray*}
Since $h_4$ is irreducible, we can directly verify that $\det A\notin (h_4)$. 

Suppose $d\geq 2$ and the statement holds for $d-1$. Suppose to the contrary that $\det A\in (h_{2d+2})$, so there exists $g\in \mathbb{K}[t_{ij}|1\leq i<j\leq 2d+3]$ such that 
\begin{equation}\label{containment}
\det A=g \  h_{2d+2}.
\end{equation}
On the one hand the degree of $\det A$ with respect to the variable 
$t_{12}$ is at most $4d-4$, and the coefficient of $t_{12}^{4d-4}$ is given by
\begin{equation*}
\resizebox{\textwidth}{!}{
$\det \begin{pmatrix}
2h_{1,2d+1,2d+2}^2 & h_{1,2,2d+2}^2+h_{1,2d+1,2d+2}^2+h_{2,2d+1,2d+2}^2\\ 
h_{1,2,2d+2}^2+h_{1,2d+1,2d+2}^2+h_{2,2d+1,2d+2}^2 &2h_{2,2d+1,2d+2}^2 
\end{pmatrix}\cdot \det A^{\prime},$
}
\end{equation*}
where  $A^{\prime}$ is the $(2d-2)\times (2d-2)$ matrix over $\mathbb{K}[t_{ij}| 3\leq i<j\leq 2d+3]$ with 
\[A^{\prime}(i, j)=h_{1,2,i+2,j+2,2d+2}^2+h_{1,2,i+2,2d+1,2d+2}^2+h_{1,2,j+2,2d+1,2d+2}^2.\]
On the other hand, the degree of $h_{2d+2}$ with respect to $t_{12}$ is $1$, and the coefficient of $t_{12}$ is $h_{1,2,2d+2}$. By comparing the  coefficient of $t_{12}^{4d-4}$, we obtain from  (\ref{containment}) the following equality:
\begin{eqnarray*}
& \resizebox{\textwidth}{!}{
$\det \begin{pmatrix}
2h_{1,2d+1,2d+2}^2 & h_{1,2,2d+2}^2+h_{1,2d+1,2d+2}^2+h_{2,2d+1,2d+2}^2\\ 
h_{1,2,2d+2}^2+h_{1,2d+1,2d+2}^2+h_{2,2d+1,2d+2}^2 &2h_{2,2d+1,2d+2}^2 
\end{pmatrix}\cdot \det A^{\prime}$
}\\
&=g_1 \ h_{1,2,2d+2},
\end{eqnarray*}
where $g_1\in \mathbb{K}[t_{ij}|1\leq i<j\leq 2d+3]$. A further simplification yields that 
\[(h_{1,2d+1,2d+2}^2-h_{2,2d+1,2d+2}^2)^2 \ \det A^{\prime}=g_2 \ h_{1,2,2d+2},\] for some $g_2\in \mathbb{K}[t_{ij}|1\leq i<j\leq 2d+3]$. Applying the induction hypothesis to the skew symmetric matrix $T(\hat{1},\hat{2})$, we obtain  
\[\det A^{\prime}\notin (h_{1,2,2d+2}).\]
Then $(h_{1,2d+1,2d+2}^2-h_{2,2d+1,2d+2}^2)^2\in (h_{1,2,2d+2}) $, since $h_{1,2,2d+2}$ is irreducible. Note the variable $t_{3,2d+1}$ is absent from  $(h_{1,2d+1,2d+2}^2-h_{2,2d+1,2d+2}^2)^2$, and the degree of $h_{1,2,2d+2}$ with respect to $t_{3,2d+1}$ is $1$. This obviously contradicts the relation $(h_{1,2d+1,2d+2}^2-h_{2,2d+1,2d+2}^2)^2\in (h_{1,2,2d+2}) $. So $\det A\notin (h_{2d+2})$, as desired. 
\end{proof}

 We consider the ring $\mathbb{K}[\mathfrak{C}^d_{\mathfrak{g}}][T_0][[ T_{i_1 \cdots i_d} | (i_1,\cdots, i_d)\in S]]$, by adding a formal variable $T_0 $.
\begin{lemma}\label{lemma:existence of tildeN}
There exists $\tilde{N}\in \mathbb{K}[\mathfrak{C}^d_{\mathfrak{g}}][T_0][[ T_{i_1 \cdots i_d} | (i_1,\cdots, i_d)\in S]]$
satisfying:
\begin{equation}\notag
\tilde{N}^2=\det (T_0^2 I_n+\sum\limits_{(i_1,\cdots, i_d)\in S} X(1)^{i_1}X(2)^{i_2}\cdots X(d)^{i_d}T_{i_1\cdots i_d}).
\end{equation} 
\end{lemma}

\begin{proof}
For ease of notation, set $R=\mathbb{K}[\mathfrak{C}^d_{\mathfrak{g}}]$. 
Note to begin with that
\begin{displaymath}
T_0^2 I_n+\sum\limits_{(i_1,\cdots, i_d)\in S} X(1)^{i_1}X(2)^{i_2}\cdots X(d)^{i_d}T_{i_1\cdots i_d}=T_0^2 I_n+\sum\limits_{j=1}^d X(j)F(j),
\end{displaymath}
where 
\begin{displaymath}
F(j)=\sum\limits_{(0,\cdots,0,i_j,\cdots, i_d)\in S, i_j\geq 1}X(j)^{i_j-1}X(j+1)^{i_{j+1}}\cdots X(d)^{i_d}T_{0\cdots0i_j\cdots i_d}.
\end{displaymath}

By Lemma \ref{lemma:pfaffian solution}, we can take a skew symmetric matrix $T\in M_{2d+3}(\mathbb{K})$ such that $h_{2d+2}=0$ and $h_{2d+3}\det A\neq 0$. Here recall $h_{i_1,\cdots, i_k}=\mathrm{Pf}\ T(\hat{i}_1,\cdots, \hat{i}_k)$, and $A\in M_{2d}(\mathbb{K})$ whose $(i,j)$-entry is $a_{ij}:=h_{i,j,2d+2}^2+h_{i,2d+1,2d+2}^2+h_{j,2d+1,2d+2}^2$. By scaling $T$ if necessary, we can assume $h_{2d+3}=1$.

Over the polynomial ring $\mathbb{K}[y,x_0, x_i| 1\leq i\leq 2d]$, let $v:=(0,\cdots, 0, 1, yx_0)$ be the $2d+3$-tuple. Define the $(2d+3)\times (2d+3)$ matrix $B(T)$ by
\begin{displaymath}
B(T):=\mathrm{diag }(yx_1, yx_2,\cdots,  yx_{2d}, \sum\limits_{i=1}^{2d}yx_i,0,0)+T,
\end{displaymath}
 and the $(2d+4)\times (2d+4)$ matrix $M(T)$ by
\begin{equation}\notag 
 M(T):=\left (\begin{array}{c|c}
 B(T)& v^{t}\\ 
 \hline 
 -v&0  
 \end{array}\right ).
 \end{equation} 
By Lemma \ref{lemma:expansion of pfaffian} below, 
 \begin{displaymath}
 \begin{split}
 &\det M(T)=(h_{2d+2}-h_{2d+3} \ yx_0)^2+\sum\limits_{i=1}^{2d}y^2x_i^2(h_{i,2d+1,2d+2}-h_{i,2d+1,2d+3} yx_0)^2\\ 
 &+\sum\limits_{1\leq i<j\leq 2d}y^2x_i x_j[(h_{i,j,2d+2}-h_{i,j,2d+3}yx_0)^2+(h_{i,2d+1,2d+2}-h_{i,2d+1,2d+3}yx_0)^2 \\ 
 &+(h_{j,2d+1,2d+2}-h_{j,2d+1,2d+3}yx_0)^2]+y^4 \ g,
 \end{split}
 \end{displaymath}
 where $g\in \mathbb{K}[y,x_0, x_i| 1\leq i\leq 2d]$ and $g\in (x_1,\cdots, x_{2d})$. Since $h_{2d+2}=0$, $h_{2d+3}=1$, and $h_{i,j,2d+2}^2+h_{i,2d+1,2d+2}^2+h_{j,2d+1,2d+2}^2=a_{ij}$, the above expression can be simplified as 
 \begin{equation}\label{equ:def of M(T)}
 \det M(T)=y^2(x_0^2+\frac{1}{2}\sum\limits_{i=1 }^{2d}\sum\limits_{j=1}^{2d}a_{ij}x_i x_j)+y^3 \ g_1,
 \end{equation}
 where $g_1\in \mathbb{K}[y,x_0, x_i| 1\leq i\leq 2d]$ and  $g_1\in (x_1,\cdots, x_{2d})$. Since $A=(a_{ij})\in M_{2d}(\mathbb{K})$ is symmetric  with $\det A\neq 0$, there exists an invertible  $P\in M_{2d}(\mathbb{K})$ such that if we let $(y_1,\cdots, y_{2d})=(x_1,\cdots, x_{2d})\cdot P$, then 
 \[\frac{1}{2}\sum\limits_{i=1 }^{2d}\sum\limits_{j=1}^{2d}a_{ij}x_i x_j=\sum\limits_{i=1}^d y_{2i-1} y_{2i}=y_1y_2+y_3y_4+\cdots +y_{2d-1}y_{2d}.\] 
 In the matrix $M(T)$, we replace $x_0$ by $T_0I_n $, $y_{2i-1}$ by the skew symmetric matrix $X(i)$, $y_{2i}$ by $F(i)$,\ $1\leq i\leq d$, and any constant number $a\in \mathbb{K}$ by $a I_n$. In this way, we obtain a $(2d+4)n\times (2d+4)n$ matrix $\tilde{M}$ over $R[y, T_0][[ T_{i_1 \cdots i_d} | (i_1,\cdots, i_d)\in S]]$. Note $\tilde{M}$ is skew symmetric.  Define 
$
\tilde{Q}_y:={\rm Pf}\ \tilde{M}\in R[y, T_0][[ T_{i_1 \cdots i_d} | (i_1,\cdots, i_d)\in S]].
$
Then  by (\ref{equ:def of M(T)}) and  Lemma \ref{lemma:determinant of a block matrix}, 
\begin{equation}\label{equ:Q_y^2}
\tilde{Q}_y^2=\det \ \tilde{M}=y^{2n}\det (T_0^2 I_n+\sum\limits_{i=1}^d X(i)F(i)+y \ g_2),
\end{equation}
where $g_2$ is an $n\times n$ matrix over  $ R[y, T_0][[ T_{i_1 \cdots i_d} | (i_1,\cdots, i_d)\in S]]$. Note $R=\oplus_{i=0}^{\infty}R_i$ is a graded $\mathbb{K}$-algebra with $R_0=\mathbb{K}$, and the entries of $X(i)$ are all in $R_1$. Since $g_1\in (x_1,\cdots, x_{2d})$, we can  see 
\[\det (T_0^2 I_n+\sum\limits_{i=1}^d X(i)F(i)+y \ g_2)-T_0^{2n}\in \oplus_{i=1}^{\infty}R_i[y, T_0][[ T_{i_1 \cdots i_d} | (i_1,\cdots, i_d)\in S]].\]  
Now in $R[[y, T_0, T_{i_1 \cdots i_d} | (i_1,\cdots, i_d)\in S]]$, we can apply Lemma \ref{lemma:f^2=y^2n g implies f=y^n h} to the equality (\ref{equ:Q_y^2})   and obtain $\tilde{N}_y\in R[[y, T_0, T_{i_1 \cdots i_d} | (i_1,\cdots, i_d)\in S]]$ satisfying $\tilde{Q}_y=y^n \tilde{N}_y$. Moreover, $\tilde{N}_y\in R[y, T_0][[ T_{i_1 \cdots i_d} | (i_1,\cdots, i_d)\in S]]$ since $\tilde{Q}_y\in R[y, T_0][[ T_{i_1 \cdots i_d} | (i_1,\cdots, i_d)\in S]]$. By (\ref{equ:Q_y^2}), we see
\begin{equation}\label{equ:tilde N^2}
\tilde{N}_y^2=\det (T_0^2 I_n+\sum\limits_{i=1}^d X(i)F(i)+y \ g_2).
\end{equation}
Now let $\tilde{N}\in R[T_0][[ T_{i_1 \cdots i_d} | (i_1,\cdots, i_d)\in S]]$ be the evaluation of $\tilde{N}_y$ at $y=0$, we deduce from (\ref{equ:tilde N^2}) the desired equation:
\begin{displaymath}
\begin{split}
\tilde{N}^2&=\det (T_0^2 I_n+\sum\limits_{i=1}^d X(i)F(i))\\
&=\det (T_0^2 I_n+\sum\limits_{(i_1,\cdots, i_d)\in S} X(1)^{i_1}X(2)^{i_2}\cdots X(d)^{i_d}T_{i_1\cdots i_d}).
\end{split}
\end{displaymath}
\end{proof}

Now we begin to complete the proof of   Lemma \ref{lemma:deg of sqrt F_g}. Still let $R=\mathbb{K}[\mathfrak{C}^d_{\mathfrak{g}}]$.

By Lemma \ref{lemma:existence of tildeN}, we can find  $\tilde{N}\in R[T_0][[ T_{i_1 \cdots i_d} | (i_1,\cdots, i_d)\in S]]$
satisfying:
\begin{equation}\label{equ:tilde N^2= det M}
\tilde{N}^2=\det (T_0^2 I_n+\sum\limits_{(i_1,\cdots, i_d)\in S} X(1)^{i_1}X(2)^{i_2}\cdots X(d)^{i_d}T_{i_1\cdots i_d}).
\end{equation} 
Let $\tilde{N}(\underline{0})\in R[T_0]$ be the  valuation of $\tilde{N}$ at $T_{i_1\cdots i_d}=0$, \ $(i_1,\cdots, i_d)\in S$. Then $\tilde{N}(\underline{0})^2=\det (T_0^2 I_n)= T_0^{2n}$. Recall that $R=\oplus_{i=0}^{\infty}R_i$ is a graded $\mathbb{K}$-algebra with $R_0=\mathbb{K}$.  By Lemma \ref{lemma:f^2 determines f}, we get $\tilde{N}(\underline{0})=\pm T_0^n$. We assume without loss of generality that $\tilde{N}(\underline{0})= T_0^n$.

Let $N \in R[[ T_{i_1 \cdots i_d} | (i_1,\cdots, i_d)\in S]]$ be the evaluation of $\tilde{N}$ at $T_0=1$. Obviously the constant coefficient of $N$ is $1$, and  
\begin{displaymath}
N^2=\det (1+\sum\limits_{(i_1,\cdots, i_d)\in S} X(1)^{i_1}X(2)^{i_2}\cdots X(d)^{i_d}T_{i_1\cdots i_d})=F_{\mathfrak{g}}.
\end{displaymath}
 By applying Lemma \ref{lemma:f^2 determines f} to $R[[ T_{i_1 \cdots i_d} | (i_1,\cdots, i_d)\in S]]$, we get $\sqrt{F_{\mathfrak{g}}}=N$. 

For any non-zero $\lambda\in \mathbb{K}^*$, 
let $\tilde{N}_{\lambda} \in R[T_0][[ T_{i_1 \cdots i_d} | (i_1,\cdots, i_d)\in S{}]]$ be the image of $\tilde{N}$ under the automorphism of $R$-algebras:  
\begin{displaymath}
\begin{split}
R[T_0][[ T_{i_1 \cdots i_d} | (i_1,\cdots, i_d)\in S{}]] & \rightarrow R[T_0][[ T_{i_1 \cdots i_d} | (i_1,\cdots, i_d)\in S]]\\
 T_0&\mapsto \lambda T_0 \\ 
T_{i_1 \cdots i_d}&\mapsto \lambda^2 T_{i_1 \cdots i_d}
\end{split}
\end{displaymath}
  By (\ref{equ:tilde N^2= det M}), $\tilde{N}_{\lambda}^2=\lambda^{2n} \tilde{N}^2=(\lambda^n \tilde{N})^2$.  Since $\tilde{N}(\underline{0})= T_0^n$, we see $\tilde{N}_{\lambda }(\underline{0})=\lambda ^n \tilde{N}(\underline{0})=\lambda ^n T_0^n$ is not a zero-divisor in $R[T_0][[ T_{i_1 \cdots i_d} | (i_1,\cdots, i_d)\in S{}]]$.  Then $\tilde{N}_{\lambda }=\lambda^n \tilde{N}$ by Lemma \ref{lemma:f^2 determines f}. From this we deduce that, by requiring $\deg T_0=1$ and  $\deg T_{i_1\cdots i_d}=2$ for all $(i_1,\cdots, i_d)\in S$, $\tilde{N}$ is a degree $n$ weighted homogeneous formal power series with respect to the variables $\{T_0, T_{i_1\cdots i_d}| (i_1,\cdots, i_d)\in S\}$. Thus $\deg \sqrt{F_{\mathfrak{g}}}=\deg N\leq \lfloor \frac{n}{2}\rfloor$, and this completes the proof of Lemma \ref{lemma:deg of sqrt F_g}.

\begin{lemma}\label{lemma:expansion of pfaffian}
Let $R$ be a $\mathbb{K}$-algebra.  Suppose  $n\geq 2$, and  $T=(t_{ij})\in M_{2n}(R)$ is a  skew symmetric matrix.  Over the polynomial ring $R[x_1,\cdots, x_{2n-3}]$, let 
\[M:=T+\mathrm{diag} (x_1,\cdots, x_{2n-3},0,0,0)\] 
be the sum of $T$ and a diagonal matrix.  Suppose $t_{i, 2n}=0$, for $i=1,\cdots, 2n-3$. Then the determinant $\det M$  has the following expansion as a polynomial in $x_1,\cdots, x_{2n-3}$:
\begin{displaymath}
\begin{split}
\det M&=(h_{2n-2,2n} ~t_{2n-2,2n}- h_{2n-1, 2n}~t_{2n-1,2n})^2\\ 
&+\sum\limits_{1\leq i<j\leq 2n-3}x_ix_j (h_{i,j,2n-2,2n}~t_{2n-2,2n}- h_{i,j,2n-1,2n}~t_{2n-1,2n})^2\\ 
&+ \sum\limits_{k=2}^{n-2}f_{2k}.
\end{split}
\end{displaymath}  
Here  for even $m$ and for  $1\leq i_1<i_2<\cdots <i_m\leq 2n$, the symbol $h_{i_1,\cdots, i_m}$ means ${\rm Pf}~ T(\hat{i}_1,\cdots, \hat{i}_m)$,  the Pfaffian of the matrix obtained from $T$ by deleting the $i_j$-th row and the $i_j$-th column for all $1\leq j\leq m$, and $f_{2k}$ means a degree $2k$ homogeneous polynomial in $x_1,\cdots, x_{2n-3}$. 
\end{lemma}

\begin{proof}
By expanding $\det M$ along the diagonal, we obtain
\begin{equation}\notag
\det M=\det T+\sum\limits_{m=1}^{2n-3}\sum\limits_{1\leq i_1<i_2<\cdots<i_{m}\leq 2n-3}\det T(\hat{i}_1,\cdots, \hat{i}_{m})\ x_{i_1}x_{i_2}\cdots x_{i_{m}}.
\end{equation}

Note when $m$ is odd, $\det T(\hat{i}_1,\cdots, \hat{i}_{m})=0$ because $T(\hat{i}_1,\cdots, \hat{i}_{m})$ is a skew symmetric matrix of odd order.  So the above expansion reduces to

\begin{equation}\label{equ:expansion of f along diagonal}
\det M=\det T+\sum\limits_{k=1}^{2n-4}\sum\limits_{i_1,\cdots, i_{2k}}\det T(\hat{i}_1,\cdots, \hat{i}_{2k})\ x_{i_1}x_{i_2}\cdots x_{i_{2k}},
\end{equation}
where the sum $\sum\limits_{i_1,\cdots, i_{2k}}$ is over
\[\{{(i_1,\cdots, i_{2k})| \ 1\leq i_1<i_2<\cdots <i_{2k}\leq 2n-3}\}.\]

 For a $2m\times 2m$ skew symmetric matrix $A=(a_{ij})$, the Pfaffian ${\rm Pf}~ A$ can be expanded with respect to the last column (cf. Cayley [\ref{Cayley}]) as follows:
\[{\rm Pf}~ A=\sum\limits_{i=1}^{2m-1}(-1)^{i+1} {\rm Pf}~ A(\hat{i},\widehat{2m})~ a_{i, 2m}.\]

Applying this expansion to the skew symmetric matrix $T(\hat{i}_1,\cdots, \hat{i}_{2k})$,
we obtain 
\begin{displaymath}
\begin{split}
{\rm Pf}~ T(\hat{i}_1,\cdots, \hat{i}_{2k})=&-t_{2n-2,2n} {\rm Pf}~T(\hat{i}_1,\cdots, \hat{i}_{2k},\widehat{2n-2},\widehat{2n})\\ 
&+t_{2n-1,2n} {\rm Pf}~T(\hat{i}_1,\cdots, \hat{i}_{2k},\widehat{2n-1},\widehat{2n}).
\end{split}
\end{displaymath}

Combining this with \eqref{equ:expansion of f along diagonal}, we get the desired expansion of $\det M$.
\end{proof}

\section{Applications}\label{sec:corollaries}
In this section, we  apply the restriction isomorphism to obtain some identities about  polynomial functions on commuting skew symmetric matrices. 

From now on,  we suppose $\rm{char}\  \mathbb{K} =0$, and let $G=O_n(\mathbb{K})$ be the orthogonal group, so that its Lie algebra $\mathfrak{g}$ is the space of  $n\times n$ skew symmetric matrices.  We also fix the Cartan subalgebra 
\[\mathfrak{t}=\{\textmd{SK }(x_1,\cdots, x_{\lfloor\frac{n}{2}\rfloor})| x_i\in \mathbb{K}, \text{ for all } \ 1\leq i\leq \lfloor\frac{n}{2}\rfloor\}\]  
as in Section \ref{subsec:proof of lemma deg F_t}, where recall $\textmd{SK }(x_1,\cdots, x_{\lfloor\frac{n}{2}\rfloor})$ is the  $n\times n$ skew symmetric matrix defined in (\ref{equ:sk(x_1,...,x_m)}). Recall  $(X(1),\cdots, X(d))\in \mathfrak{C}^d_{\mathfrak{g}}(\mathbb{K}[\mathfrak{C}^d_{\mathfrak{g}}])$ is the tautological $\mathbb{K}[\mathfrak{C}^d_{\mathfrak{g}}]$-valued  point of $\mathfrak{C}^d_{\mathfrak{g}}$.  Under the restriction homomorphism $M_n( \mathbb{K}[\mathfrak{C}^d_{\mathfrak{g}}] )\rightarrow M_n(\mathbb{K}[\mathfrak{t}^d])$, the  skew symmetric  matrices $X(i)$ are mapped to $Y(i)$.

Let $R$ be a  $\mathbb{K}$-algebra.
\begin{corollary}
Suppose  $n\geq 3$ is odd and let $X_1,\cdots, X_d \in M_n(R)$ be commuting skew symmetric matrices. For any $f\in \mathbb{K}[x_1,\cdots, x_d] $, if $f(0,\cdots, 0)=0$, then $\det f(X_1,\cdots, X_d)=0$.   
\end{corollary}
\begin{proof}
Clearly there exists a $\mathbb{K}$-algebra homomorphism $\varphi: \mathbb{K}[\mathfrak{C}_{\mathfrak{g}}^d]\rightarrow R$ such that $\varphi $ maps the matrix  $X(i)$ to $ X_i$, \ $i=1,\cdots, d$.  Then $\det f(X_1,\cdots, X_d)=\varphi(\det f(X(1),\cdots, X(d)))$. Since $\det f(X(1),\cdots, X(d))\in \mathbb{K}[\mathfrak{C}_{\mathfrak{g}}^d]^G$, we can  apply  the restriction homomorphism  
\[\Phi: \mathbb{K}[\mathfrak{C}_{\mathfrak{g}}^d]^G\rightarrow \mathbb{K}[\mathfrak{t}^d]^W\] 
to obtain 
\begin{displaymath}
\Phi(\det f(X(1),\cdots, X(d)))=\det f(Y(1),\cdots, Y(d)).
\end{displaymath}
By the definition of $Y(i)$ we see $\det f(Y(1),\cdots, Y(d))=0$. Since $\Phi$ is an isomorphism by Theorem \ref{main thm: char.  0}, we obtain $\det f(X(1),\cdots, X(d))=0$ and hence $\det f(X_1,\cdots, X_d)=0$.
\end{proof}

\begin{corollary}\label{cor:multiplicative of Pfaffian}
Suppose $n\geq 2$ is even. Let $X_1, X_2, X_3\in M_n(R)$ be commuting skew symmetric matrices. Then 
\begin{displaymath}
\mathrm{Pf }(X_1X_2X_3)=(-1)^{\frac{n}{2}}\mathrm{Pf }(X_1) \mathrm{Pf }(X_2) \mathrm{Pf }(X_3).
\end{displaymath}
\end{corollary}
\begin{proof}
The  case $n=2$ is trivial, so we assume $n\geq 4$ and take $d=3$. Let $G ^{\prime}=SO_n(\mathbb{K})$ be the special orthogonal group, which has the same Lie algebra as $G=O_n(\mathbb{K})$. 
There exists a $\mathbb{K}$-algebra homomorphism $\varphi: \mathbb{K}[\mathfrak{C}_{\mathfrak{g}}^d]\rightarrow R$ such that $\varphi(X(i))=X_i$, \ $i=1,2,3$. Then 
\begin{displaymath}
\begin{split}
&\mathrm{Pf }(X_1X_2X_3)-(-1)^{\frac{n}{2}}\mathrm{Pf }(X_1) \mathrm{Pf }(X_2) \mathrm{Pf }(X_3)\\ 
&=\varphi (\mathrm{Pf }(X(1)X(2)X(3))-(-1)^{\frac{n}{2}}\mathrm{Pf }(X(1)) \mathrm{Pf }(X(2)) \mathrm{Pf }(X(3))).
\end{split}
\end{displaymath}
Let $r=\mathrm{Pf }(X(1)X(2)X(3))-(-1)^{\frac{n}{2}}\mathrm{Pf }(X(1)) \mathrm{Pf }(X(2)) \mathrm{Pf }(X(3))$, we see $r\in \mathbb{K}[\mathfrak{C}_{\mathfrak{g}}^d]^{G^{\prime}}$. Then we apply  the  restriction homomorphism  $\Phi:\mathbb{K}[\mathfrak{C}_{\mathfrak{g}}^d]^{G^{\prime}} \rightarrow \mathbb{K}[\mathfrak{t}^d]^{W^{\prime}}$ to obtain 
\begin{displaymath}
\Phi(r)=\mathrm{Pf }(Y(1)Y(2)Y(3))-(-1)^{\frac{n}{2}}\mathrm{Pf }(Y(1)) \mathrm{Pf }(Y(2)) \mathrm{Pf }(Y(3)).
\end{displaymath}
By direct computations, we can see $\Phi(r)=0$. Since $\Phi$ is an isomorphism by Theorem \ref{thm:main thm in the n even G=SO_n case}, we get  finally $r=0$.
\end{proof}

For a positive integer $m$, write $\mathcal{P}_m$ for the set of partitions $\lambda=\lambda_1\cup\cdots  \cup\lambda_h$ of the set $\{1,\cdots, m\}$ into the disjoint union of non-empty subsets $\lambda_i$, and denote $h(\lambda)=h$ the number of parts of the partition $\lambda$.  

\begin{corollary}\label{cor:trace identity}
Suppose $n\geq 2$, $d\geq 1$. Let $m=\lfloor \frac{n}{2}\rfloor+1$. Suppose  $X_1, \cdots, X_d\in M_n(R)$ are commuting skew symmetric matrices. For $j=1,\cdots, m$, let $Y_j=\prod_{i=1}^d X_i^{a_{ij}}\in M_n(R)$ be a monomial of $X_1,\cdots, X_d$, with $\sum_{i=1}^d a_{ij}>0$ even.  Then 
\begin{displaymath}
\sum\limits_{\lambda\in \mathcal{P}_{m}}(\frac{-1}{2})^{h(\lambda)}\prod\limits_{i=1}^{h(\lambda)}
((|\lambda_i|-1)! \cdot \mathrm{ tr }\prod \limits_{s\in \lambda_i} Y_s)=0.
\end{displaymath}
\end{corollary}
\begin{proof}
In a similar way as above, it suffices  to verify the identity under the assumption   $R=\mathbb{K}[\mathfrak{t}^d]$ and  $X_i=Y(i)$, \ $i=1,\cdots, d$. Then the required trace identity is just a reformulation of Domokos \cite[Proposition 2.3]{Domokos}.
\end{proof}

\begin{example}
As an illustration, we take $n=4$, so $m=3$. All of the partitions of $\{1,2,3\}$ are:
\begin{displaymath}
\{1,2,3\}, \ \{1,2\}\cup \{3\}, \ \{1,3\}\cup \{2\}, \ \{2,3\}\cup \{1\}, \ \{1\}\cup \{2\}\cup\{3\}.
\end{displaymath}
Then according to Corollary \ref{cor:trace identity}, for any $d\geq 1$ commuting skew symmetric $4\times 4$ matrices $X_1,\cdots, X_d\in M_4(R)$, and for  $Y_1, Y_2, Y_3$ which are monomials of $X_i$ of even degree, we have the following trace identity:
\begin{eqnarray*}
&2\rm{tr}(Y_1) \rm{tr}(Y_2Y_3)+2\rm{tr}(Y_2 )\rm{tr}(Y_1Y_3)+2\rm{tr}(Y_3)\rm{tr}(Y_1Y_2)\\
&=8\rm{tr}(Y_1Y_2Y_3)+\rm{tr}(Y_1) tr(Y_2) tr(Y_3).
\end{eqnarray*}
To be more specific, let $Y_i=X_i^2$, \ $i=1,2,3$, then this identity reduces to:
\begin{eqnarray*}
&2\rm{tr}(X_1^2) tr(X^2_2X^2_3)+2\rm{tr}(X^2_2) tr(X^2_1X^2_3)+2\rm{tr}(X^2_3) tr(X^2_1X^2_2)\\ 
&=8\rm{tr}(X^2_1X^2_2X^2_3)+\rm{tr}(X^2_1) tr(X^2_2) tr(X^2_3).
\end{eqnarray*}
\end{example}

\end{document}